\documentclass[12pt,a4paper]{amsart}
\calclayout
\numberwithin{equation}{section}
\allowdisplaybreaks
\theoremstyle{plain}
\newtheorem{thm}{Theorem}[section]
\newtheorem{lem}[thm]{Lemma}
\newtheorem{prop}[thm]{Proposition}

\theoremstyle{definition}
\newtheorem*{acknowledgment*}{Acknowledgment}
\theoremstyle{remark}

\usepackage{amssymb}
\usepackage{mathtools}
\mathtoolsset{showonlyrefs}
\usepackage{mathrsfs}
\usepackage{comment}
\usepackage{hyperref}
	\def\MR#1{\href{http://www.ams.org/mathscinet-getitem?mr=#1}{MR-#1}}
	
	\def\DOI#1{\href{https://doi.org/#1}{doi:#1}}
\begin{document}
\title{Quasi-stationary distributions for subcritical superprocesses}
\author[R. Liu, Y.-X. Ren, R. Song and Z. Sun]{Rongli Liu, Yan-Xia Ren, Renming Song and Zhenyao Sun}
\address{Rongli Liu\\ School of Science\\ Beijing Jiaotong University\\ Beijing 100044\\ P. R. China}
\email{rlliu@bjtu.edu.cn}
\thanks{The research of Rongli Liu is supported in part by NSFC (Grant No. 11301261), and the Fundamental Research Funds for the Central Universities (Grant No.  2017RC007)}
\address{Yan-Xia Ren\\ LMAM School of Mathematical Sciences \& Center for
Statistical Science\\ Peking University\\ Beijing 100871\\ P. R. China}
\email{yxren@math.pku.edu.cn}
\thanks{The research of Yan-Xia Ren is supported in part by NSFC (Grant Nos. 11671017 and 11731009)  and LMEQF}
\address{Renming Song\\ Department of Mathematics\\ University of Illinois at Urbana-Champaign \\ Urbana \\ IL 61801\\ USA}
\email{rsong@illinois.edu}
\thanks{The research of Renming Song is supported in part by a grant from the Simons Foundation (\#429343, Renming Song)}
\address{Zhenyao Sun\\ Faculty of Industrial Engineering and Management \\ Technion, Isreal Institute of Technology \\ Haifa 3200003\\ Isreal}
\email{zhenyao.sun@gmail.com}
\begin{abstract}
    Suppose that $X$ is a subcritical superprocess.
	Under some asymptotic conditions on the mean semigroup of $X$, we prove the Yaglom limit of $X$ exists and identify all  quasi-stationary distributions of $X$.
\end{abstract}
\maketitle
\section{Introduction}
\subsection{Background}\label{sec:BGD}
	Denote $\mathbb Z_+:= \{1,2,\cdots\}$ and $\mathbb N = \mathbb Z_+ \cup \{0\}$.
	Suppose that $Z=\{(Z_n)_{n\in \mathbb N}; (P_z)_{z\in \mathbb N}\}$
	is a Galton-Watson process with offspring distribution
		$(p_n)_{n\in \mathbb N}$.
	Let $m:=\sum^{\infty}_{n=1}np_n$ be the mean
	of the offspring distribution.
	 It is well known that when $m \leq 1$ and $p_1<1$,
	the process $Z$ becomes extinct in finite time almost surely, that is,
\[
	P_z(Z_n = 0 \text{ for some $n \in \mathbb N$} ) = 1, \quad z \in \mathbb N.
\]
	Let $\zeta:=\inf\{n\geq 0: Z_n=0\}$ be the extinction time of $Z$.
	If $\nu$ is a distribution on
	$\mathbb Z_+$ such that for any $z\in \mathbb Z_+$ and subset $A$ of $\mathbb Z_+$,
\[
	\lim_{n\rightarrow\infty} P_z\left(Z_n\in A\middle|\zeta>n \right)=\nu(A),
\]
	then we say that $\nu$ is the Yaglom limit of $Z$.
	Yaglom \cite{Yaglom47} showed that such limit exists when $m < 1$ and the offspring distribution has finite second moment.
	This was generalized to the case without the second moment assumption
	in \cite{Heathcote, Joffe1967On}.
	See also \cite[pp. 64--65]{AthreyaNey1972Branching} for an alternative analytical approach; and \cite{LyonsPemantlePeres1995Conceptual} for a probabilistic proof.
	If $\nu$ is a distribution on $\mathbb Z_+$ such that for any subset $A$ of $\mathbb Z_+$,
\[
	\sum_{z = 1}^\infty \nu(z) P_z\left(Z_n\in A\middle|\zeta>n \right)
	=\nu(A), \quad n \in \mathbb N,
\]
	then we say $\nu$ is a quasi-stationary distribution of $Z$.
	Hoppe and Seneta \cite{HoppeSeneta1978Analytical} studied the quasi-stationary distributions of $(Z_n)_{n\in \mathbb N}$.
	Recently, Maillard \cite{Maillard2018The} characterized all  $\lambda$-invariant measures of $(Z_n)_{n \in \mathbb N}$.
	If a $\lambda$-invariant measure is a probability measure, then it is equivalent to a quasi-stationary distribution.
	Multitype analogs for the Yaglom limit results can be found in
	\cite{Hoppe1975Stationary, HoppeSeneta1978Analytical, JoffeSpitzer1967On}.

	Now suppose that $Z =\{(Z_t)_{t \geq 0}; (P_x)_{x\geq 0}\}$  is a  continuous-state branching process on $[0,\infty)$ where $0$ is an absorbing state.  Let
    $\zeta:=\inf\{t\geq 0: Z_t=0\}$ be the extinction time of $Z$.
	If $\nu$ is a distribution on $(0,\infty)$ such that for any $x>0$ and Borel subset $A$ of $(0,\infty)$,
\[
\lim_{t\rightarrow\infty} P_x\left(Z_t\in A\middle|\zeta>t\right)=\nu(A),
\]
	then $\nu$ is called the Yaglom limit of $Z$.
	If $\nu$ is a distribution on $(0,\infty)$ such that for any Borel subset $A$ of $(0,\infty)$,
\[
\int_{(0,\infty)} \nu(dx) P_x (Z_t \in A | \zeta > t) = \nu (A),
	\quad t\geq 0,
\]
	then we say $\nu$ is a quasi-stationary distribution for $Z$.
	The Yaglom limits of continuous-state branching processes were studied in \cite{Li00}, where conditioning of the type $\{\zeta>t+r\}$ for any finite $r>0$
    instead of $\{\zeta>t\}$ was also considered.
	Lambert \cite{Lambert2007Quasi-stationary} also studied Yaglom limits using a different method, and characterized all the quasi-stationary distributions  for $Z$.
	Seneta and Vere-Jones \cite{SenetaVere-Jones1968On} studied some similar type of conditional limits for discrete-time continuous-state branching processes.
	Recently \cite{Labbe2013Quasi-stationary} considered quasi-stationary distributions for continuous-state branching processes conditioned on non-explosion.

	Asmussen and Hering \cite{AH} studied limit behaviors of subcritical branching Markov processes.
	They proved that the Yaglom limits for a class of subcritical  branching Markov processes exist under some conditions on the mean semigroup, and
    characterized all of their quasi-stationary distributions, see \cite[Chapter 5]{AH} and the references therein.

	In this paper, we are interested in a class of subcritical $(\xi, \psi)$-superprocesses.
	We will prove the existence of the Yaglom limit and identify all  quasi-stationary distributions under some asymptotic conditions on its mean semigroup.
	Our superprocesses are general in the sense that the spatial motion $\xi$ can be a general Borel right process taking values in a Polish space,
	and the branching mechanism $\psi$ can be spatially inhomogeneous.
	Precise statements of the assumptions and the results are presented in the next subsection.
	
	As far as we know, there are no results on
	Yaglom limit and quasi-stationary distributions for general superprocesses in the  literature.
	Here we list some papers dealing with superprocesses conditioning on various kinds of survivals under different settings:
	\cite{ChampagnatRaelly2008Limit, Etheridge2003A-decomposition, Evans1992The-entrance, EvansPerkins1990Measure-valued, LiuRen2009Some, RenSongSun2019Spine, RenSongSun2018Limit, RenSongZhang2015Limit, Serlet1996Occupation}.

\subsection{Main result} \label{sec:super}

	We first recall some basics about superprocesses.
	Let $E$ be a Polish space.
    Let $\partial$ be an isolated point not contained in $E$ and $E_\partial: = E \cup \{\partial\}$.
	Denote by $\mathcal B(E, D)$ the collection of Borel maps  from $E$ to some measurable space $D$.
    If $D$ is a subset of $\mathbb R$, we denote by $\mathcal B_b(E,D)$ the bounded measurable functions from $E$ to $D$.
	Assume that \emph{the underlying process} $\xi = \{(\xi_t)_{t\ge0}; (\Pi_x)_{x\in E}\}$ is an $E_\partial$-valued Borel right process with $\partial$ as an
    absorbing state.
	Denote by $\zeta:=\inf\{t>0: \xi_t=\partial\}$ the lifetime of $\xi$.
	Let \emph{the branching mechanism} $\psi$ be a function on $E \times [0,\infty)$ given by
\begin{equation}
	\psi(x,z)
	= -\beta(x) z + \sigma(x)^2 z^2 + \int_{(0,\infty)} (e^{-zu} -1 + zu) \pi(x,du),
	\quad x\in E, z\geq 0,
\end{equation}
	where $\beta, \sigma \in \mathcal B_b(E,\mathbb R)$ and $(u \wedge u^2) \pi(x,du)$ is a bounded kernel from $E$ to $(0,\infty)$.
	Let $\mathcal M_f(E)$ denote the space of all finite Borel measures on $E$ equipped with the topology of weak convergence.
	Denote by $\mathcal B(\mathcal M_f(E))$ the Borel $\sigma$-field generated by this topology.
	For any $\mu\in\mathcal M_f(E)$ and $g \in \mathcal B(E, [0,\infty))$, we use $\mu(g)$ to
   denote the integration of $g$ with respect to $\mu$ whenever the integration is well defined.
    We will use $\|\mu\|$ to denote $\mu(1)$.
  	For any $f \in \mathcal B_b(E, [0,\infty))$, there is a unique locally bounded non-negative map $(t,x)\mapsto V_tf(x)$ on $[0,\infty) \times E$ such that
\begin{equation} \label{eq:BGD.1}
	V_tf(x) + \Pi_x \left[ \int_0^{t\wedge \zeta} \psi \left(\xi_s, V_{t-s}f(\xi_s)\right) ds\right]
	= \Pi_x\left[ f(\xi_t)  \mathbf 1_{t < \zeta}\right], \quad t\geq 0, x\in E.
\end{equation}
	Here, the local boundedness of the map $(t,x) \mapsto V_tf(x)$ means that $\sup_{0\leq t\leq T, x\in E} V_tf(x)< \infty$ for $T >0$.
	Moreover, there exists an $\mathcal M_f(E)$-valued Borel right process $X =\{(X_t)_{t\geq 0}; (\mathbb P_\mu)_{\mu \in \mathcal M_f(E)}\}$ such that
\begin{equation}
	\mathbb P_\mu[e^{- X_t(f)}]
	= e^{- \mu(V_tf)},
	\quad t\geq 0,~\mu \in \mathcal M_f(E), f \in \mathcal B_b(E,[0,\infty)).
\end{equation}
	We call $X$ a \emph{$(\xi, \psi)$-superprocess}.
	See \cite{Li2011MeasureValued} for more details.

	The \emph{mean semigroup} $(P_t^\beta)_{t\geq 0}$ of $X$ is defined by
\begin{equation}
	P_t^\beta f(x)
	:= \Pi_x\left[e^{\int_0^t \beta(\xi_r)dr }f(\xi_t)  \mathbf 1_{t < \zeta}\right],
	\quad f\in \mathcal B_b(E,\mathbb R), t\geq 0, x\in E.
\end{equation}
	It is well-known (see \cite[Proposition 2.27]{Li2011MeasureValued}) that
\begin{equation} \label{Fact:M!}
	\mathbb P_\mu[X_t(f)] = \mu (P_t^\beta f),
	\quad \mu \in \mathcal M_f(E), t\geq 0, f \in \mathcal B_b(E,\mathbb R).
\end{equation}	

	In this paper, we will always assume that there exist a constant $\lambda<0$, a function $\phi \in \mathcal B_b(E,(0,\infty))$ and a probability measure
    $\nu$ with full support on $E$ such that for each $t\geq 0$, $P_t^\beta \phi = e^{\lambda t}\phi$, $\nu P_t^\beta = e^{\lambda t} \nu$ and $\nu(\phi) = 1$.
	The assumption $\lambda<0$ says that the mean of $(X_t(\phi))_{t\geq 0}$
	decays exponentially with rate $\lambda$, and in this case the superprocess $X$ is called subcritical.
	Denote by $L_1^+(\nu)$ the collection of non-negative Borel functions on $E$ which are integrable with respect to the measure $\nu$.
	We further assume that the following two conditions  hold:
\begin{equation}
\label{asp:H2!} \tag{H1}
\begin{minipage}{0.9\textwidth}
	For all $t>0$, $x\in E$, and $f\in L_1^+(\nu)$, it holds that
	 \[ P_t^\beta f(x) = e^{\lambda t} \phi(x) \nu(f) (1+ H_{t,x,f})\]
		for some real $H_{t,x,f}$ with
	\[\sup_{x\in E, f\in L_1^+(\nu)} |H_{t,x,f}| < \infty
	\mbox{ and }
	\lim_{t\to \infty} \sup_{x\in E, f\in L_1^+(\nu)} |H_{t,x,f}| = 0.\]
\end{minipage}
\end{equation}
\begin{equation}
\label{asp:H4!} \tag{H2}
\begin{minipage}{0.9\textwidth}
	There exists  $T\geq 0$ such that $\mathbb P_\nu(\|X_t\| = 0)>0$ for all $t> T$.
\end{minipage}
\end{equation}
	Note that $L_1^+(\nu)$ in \eqref{asp:H2!} can be replaced by the collection of all non-negative Borel functions $f$ with $\nu(f) = 1$.
	In fact, for any $f\in L_1^+(\nu)$ and $k \in (0,\infty)$, it is easy to see that $H_{t,x,f} = H_{t,x,kf}$.

	\eqref{asp:H2!} is mainly concerned with the spatial motion and \eqref{asp:H4!} is mainly about the branching mechanism of the superprocess.
	In Subsection \ref{subsec:examples}, we will give examples satisfying these two assumptions.

    We  mention here that quantities like $H_{t,x,f}$ in this paper might depend on the underlying process $\xi$ and the branching mechanism $\psi$.
    Since $\xi$ and $\psi$ are fixed, dependence on them will not be explicitly specified.

	Denote by $\mathbf 0$ the null measure on $E$.
	Write $\mathcal M_f^o(E) := \mathcal M_f(E)\setminus \{ \mathbf 0\}$.
	Any probability measure $\mathbf P$ on $\mathcal M_f^o(E)$ will also be understood as its unique extension on $\mathcal M_f(E)$ with  $\mathbf P(\{\mathbf
    0\}) = 0$.
	Since $\phi$ is strictly positive, we have
\begin{equation}
	\mathbb P_\mu[X_t(\phi)]
	\overset{\text{\eqref{Fact:M!}}}= \mu(P_t^\beta \phi)
	=e^{\lambda t}\mu(\phi)>0, \quad t\geq 0, \mu \in \mathcal M_f^o(E).
\end{equation}
	Thus,
\begin{equation}  \label{lem:Nd!}
		\mathbb P_\mu(\|X_t\| > 0) > 0,\quad t\geq 0,\mu \in \mathcal M_f^o(E).
\end{equation}
   Hence we can condition the superprocess $X$ on survival up to time $t$ if  the distribution of $X_0$ is not concentrated on $\{\mathbf 0\}$.
	Our first main result is the following.
	
\begin{thm} \label{Theorem:Y:H1:H2:H3:H4}
	If \eqref{asp:H2!} and \eqref{asp:H4!} hold, then there exists a probability measure $\mathbf Q_\lambda$
	on $\mathcal M_f^o(E)$ such that
\begin{equation}
 	\mathbb P_\mu \left(X_t \in \cdot \middle| \|X_t\| > 0 \right) \xrightarrow[t\to \infty]{w} \mathbf Q_\lambda(\cdot),
 	\quad \mu \in \mathcal M_f^o(E),
\end{equation}
	where $\xrightarrow{w}$ stands for weak convergence.
\end{thm}

	Now we introduce the concepts of \emph{quasi-limiting distribution (QLD)} and \emph{quasi-stationary distribution (QSD)} for our superprocess $X$.
	For any probability measure $\mathbf P$ on $\mathcal M_f(E)$, define $(\mathbf P\mathbb P)[\cdot] := \int_{\mathcal M_f(E)} \mathbb P_\mu[\cdot] \mathbf
    P(d\mu)$.
	We say a probability measure $\mathbf Q$ on $\mathcal M^o_f(E)$ is a QLD of $X$, if there exists a probability measure $\mathbf P$ on
	$\mathcal M_f^o(E)$ such that
\[
	(\mathbf P\mathbb P)\left(X_t \in B \middle| \|X_t\|>0\right) \xrightarrow[t\to \infty]{} \mathbf Q(B),
	\quad B\in \mathcal B(\mathcal M^o_f(E)).
\]
	We say a probability measure $\mathbf Q$ on $\mathcal M^o_f(E)$ is a QSD of $X$, if
\[
	(\mathbf Q \mathbb P) \left( X_t \in B \middle | \|X_t\|>0 \right) = \mathbf Q(B),
	\quad t\geq 0, B \in \mathcal B(\mathcal M^o_f(E)).
\]
	
	It follows from \cite[Proposition 1]{MeleardVillemonais2012Quasi-stationary} that, for any Markov process on $[0,\infty)$ with $0$ as an absorbing state, its QLDs and QSDs are equivalent.
	We claim that this is also the case for our $\mathcal M_f(E)$-valued Markov process $X$, for which the null measure $\mathbf 0$ is an absorbing state.
	In fact, since $E$ is a Polish space, $\mathcal M_f(E)$ is again Polish \cite[Lemma 4.3]{Kallenberg2017Random}.
	So is $\mathcal M^o_f(E)$ \cite[Theorem A1.2]{Kallenberg2002Foundations}.
	Thus $\mathcal M^o_f(E)$ is Borel isomorphic to $(0,\infty)$ \cite[Theorem A.1.6]{Kallenberg2002Foundations}.
	That is, there exists a bijection $\tau: \mathcal M^o_f(E) \to (0,\infty)$ such that both $\tau$ and its inverse $\tau^{-1}$ are Borel measurable.
	Extend $\tau$ uniquely so that it is a bijection between $\mathcal M_f(E)$ and $[0,\infty)$.
	Then, it is easy to verify that $\tau$ is a Borel isomorphism
	between $\mathcal M_f(E)$ and $[0,\infty)$ which maps $\mathbf 0$ to $0$.
	Now for any $\mathcal M_f(E)$-valued Markov process with $\mathbf 0$ as an absorbing state,
	its image under $\tau$ is a $[0,\infty)$-valued Markov process with $0$ as an absorbing state.
	Therefore we can apply \cite[Proposition 1]{MeleardVillemonais2012Quasi-stationary} to $(\tau(X_t))_{t\geq 0}$ which gives that
	a probability $\mathbf Q$ on $\mathcal M^o_f(E)$ is a QLD for $X$ if and only if it is a QSD for $X$.
	Similarly, we can apply \cite[Proposition 2]{MeleardVillemonais2012Quasi-stationary} to $X$ which says that
\begin{equation} \label{eq:S.2}
\begin{minipage}{0.9\textwidth}
	if a probability measure $\mathbf Q$ on $\mathcal M^o_f(E)$ is a QSD of $X$, then there exists an $r\in (-\infty,0)$ such that $(\mathbf Q\mathbb
    P)(\|X_t\|>0) = e^{rt}$ for all $t\geq 0$.
	In this case, we call $r$ the mass decay rate of $\mathbf Q$.
\end{minipage}
\end{equation}
	
\begin{thm} \label{thm:QSD}
	Suppose that \eqref{asp:H2!} and \eqref{asp:H4!} hold.
	Then (1) for each $r \in [\lambda, 0)$, there exists a unique QSD for $X$ with mass decay rate $r$;
	and (2) for each $r\in (-\infty, \lambda)$, there is no QSD for $X$ with mass decay rate $r$.
\end{thm}

 \subsection{Examples} \label{subsec:examples}

	In this subsection, we will give some examples satisfying \eqref{asp:H2!} and \eqref{asp:H4!}.

	We first give an example satisfying \eqref{asp:H4!}.
	Suppose that $\psi$ is bounded from below by a spatially independent branching mechanism, that is, there is a function $\widetilde\psi$ of the form
\[
	\widetilde\psi(z)
	=\widetilde\beta z+\widetilde\sigma^2z^2+\int^\infty_0(e^{-zu}-1+zu)\widetilde\pi(du), \quad z\ge 0
\]
    with $\widetilde\beta\in \mathbb{R}$, $\widetilde\sigma\ge 0$ and $\widetilde\pi$ is a measure
    on $(0, \infty)$ satisfying $\int^\infty_0(u\wedge u^2)\widetilde\pi(du)<\infty$ such that
\[
	\psi(x, z)\ge \widetilde\psi(z), \quad x\in E, z\ge 0.
\]
    If $\widetilde\psi(\infty)=\infty$ and $\int^\infty1/\widetilde\psi(z)dz<\infty$, then by
    \cite[Lemma 2.3]{RenSongZhang2015Limit}, for any $t>0$,
\[
		   \inf_{x\in E}\mathbb P_{\delta_x}(\|X_t\| = 0)>0.
\]
     Using this and \eqref{eq:OY.1} below one can easily get that
     $\mathbb P_\nu(\|X_t\| = 0)>0$ for all $t>0$. Thus \eqref{asp:H4!} is satisfied with $T=0$.

	Now we give conditions that imply \eqref{asp:H2!}.
	We assume that $\xi$ is a Hunt process and there exist an $\sigma$-finite measure $m$ with full support on $E$ and a family of strictly positive, bounded
    continuous functions $\{p_t(\cdot,\cdot): t>0\}$ on $E\times E$ such that
\begin{align}
	\Pi_x[f(\xi_t) \mathbf 1_{t< \zeta}] = \int_E p_t(x,y) f(y)m(dy), & \quad t>0, x\in E, f\in \mathcal B_b(E,\mathbb R);
	\\ \int_E p_t(x,y) m(dx) \leq 1, &\quad t>0, y\in E;
	\\ \int_E \int_E p_t(x,y)^2 m(dx)m(dy) < \infty, &\quad t>0;
\end{align}
	and the functions $x \mapsto \int_E p_t(x,y)^2m(dy)$ and $y\mapsto \int_E p_t(x,y)^2m(dx)$ are both continuous.
	Choose an arbitrary $ \mathfrak b\in \mathcal B_b(E,\mathbb R)$.
	Denote by $(P_t^\mathfrak b)_{t\geq 0}$ a semigroup of operators on $\mathcal B_b(E,\mathbb R)$ given by
\begin{equation}
	P_t^\mathfrak b f(x)
	:= \Pi_x[e^{\int_0^t \mathfrak b(\xi_s)ds} f(\xi_t) \mathbf 1_{t< \zeta}],
	\quad f\in \mathcal B_b(E, \mathbb R), t\geq 0, x\in E.
\end{equation}
	Let us write $\langle f,g \rangle_m:= \int_E f(x)g(x) m(dx)$ for  the inner product of the Hilbert space $L^2(E,m)$.
	Then it is proved in \cite{RenSongZhang2015Limit, RenSongZhang2017Central} that there exists a family of strictly positive, bounded continuous functions
    $\{p_t^\mathfrak b: t> 0\}$ on $E\times E$ such that
\begin{equation} \label{eq:IU.0}
	e^{-\|\mathfrak b\|_\infty t} p_t(x,y)
	\leq p_t^\mathfrak b(x,y) \leq e^{\|\mathfrak b\|_\infty t}p_t(x,y),
	\quad t>0, x,y\in E
\end{equation}
	and that
\begin{equation}
	P_t^\mathfrak b f(x)
	= \int_E p_t^\mathfrak b(x,y) f(y) m(dy),
	\quad t>0, x\in E.
\end{equation}
	Define the dual semigroup $(\widehat {P^{\mathfrak b}_t} )_{t\geq 0}$ by
\begin{equation}
	\widehat {P_0^{\mathfrak b}}
	= I;
	\quad \widehat {P_t^{\mathfrak b}} f(x)
	:= \int_E p_t^\mathfrak b(y,x) f(y) m(dy),
	\quad t>0,x\in E, f\in \mathcal B_b(E,\mathbb R).
\end{equation}
	It is proved in \cite{RenSongZhang2015Limit, RenSongZhang2017Central} that both $(P_t^\mathfrak b)_{t\geq 0}$ and $(\widehat {P_t^\mathfrak b})_{t\geq 0}$
    are strongly continuous semigroups of compact operators on $L^2(E,m)$.	
	Let $L^\mathfrak b$ and $\widehat {L^\mathfrak b}$ be the generators of the semigroups of compact operators on $(P_t^\mathfrak b)_{t\geq 0}$ and $(\widehat
    {P_t^\mathfrak b})_{t\geq 0}$, respectively.
	Denote by $\sigma(L^\mathfrak b)$ and $\sigma(\widehat{L^\mathfrak b})$ the spectra of $L^\mathfrak b$ and $\widehat {L^{\mathfrak b}}$, respectively.
	According to Theorem 29 of \cite{Schaefer1974Banach}, $\lambda_\mathfrak b:= \sup \Re(\sigma(L^\mathfrak b)) = \sup \Re(\sigma( \widehat{L^\mathfrak b})) $
    is a common eigenvalue of multiplicity $1$ for both $L^\mathfrak b$ and $\widehat {L^{\mathfrak b}}$.
	By the argument in \cite{RenSongZhang2015Limit} and \cite{RenSongZhang2017Central}, the eigenfunctions $h_\mathfrak b$ of $L^\mathfrak b$ and $\widehat
    h_\mathfrak b$ of $\widehat{L^\mathfrak b}$ associated with the eigenvalue $\lambda_\mathfrak b$ can be chosen to be strictly positive and continuous everywhere on $E$.
	Setting $\langle h_\mathfrak b,h_\mathfrak b\rangle_m = \langle h_\mathfrak b, \widehat h_\mathfrak b\rangle_m = 1$ so that $h_\mathfrak b$ and $\widehat
    h_\mathfrak b$ are uniquely determined pointwisely.

	We assume further that $h_0:= h_\mathfrak b|_{\mathfrak{b} \equiv 0}$ is bounded, and the semigroup $(P_t)_{t\geq 0}$ is intrinsically ultracontractive in the following sense: for all $t>0$ and $x, y \in E$, it holds that $p_t(x,y) = c_{t,x,y} h_0(x) \widehat h_0(y)$ for some positive $c_{t,x,y}$ with $\sup_{x,y \in E} c_{t,x,y}< \infty$.
	Here, $\widehat h_0 := \widehat h_\mathfrak b|_{\mathfrak{b}\equiv 0}$.
	Then, it is proved in \cite{RenSongZhang2015Limit, RenSongZhang2017Central} that, for arbitrary $\mathfrak b \in \mathcal B_b(E,\mathbb R)$, $h_\mathfrak b$ is also bounded; and $(P_t^\mathfrak b)_{t\geq 0}$ is also intrinsically ultracontractive, in the sense that for any $t> 0$ and $x,y \in E$ we have
\begin{equation} \label{eq:IU.1}
	p^\mathfrak b_t(x,y)
		  = C^1_{\mathfrak b,t,x,y}h_\mathfrak b(x) \widehat h_\mathfrak b (y)
	\end{equation}
		 for some positive $C^1_{\mathfrak b,t,x,y}$ with $\sup_{x,y \in E} C^1_{\mathfrak b,t,x,y}< \infty$.
It follows from \cite[Proposition 2.5 and Theorem 2.7]{KimSong2008Intrinsic}, when \eqref{eq:IU.1} holds, $C^1_{\mathfrak b,t,x,y}$ can be chosen so that
\begin{equation} \label{eq:IU.11}
		\sup_{x,y \in E} (C^1_{\mathfrak b,t,x,y})^{-1}
	< \infty,
	\quad t>0,
\end{equation}
	and that for any $t>0, x,y \in E$,
\begin{equation}\label{eq:IU.2}
		C^1_{\mathfrak b,t,x,y}
	= e^{t\lambda_\mathfrak{b}} (1+ C^2_{\mathfrak b,t,x,y})
\end{equation}
		for some real $C^2_{\mathfrak b,t,x,y}$ with $\lim_{t\to \infty} \sup_{x,y \in E} C^2_{\mathfrak b,t,x,y} =0$.
	Therefore,
\begin{align}
		& m(\widehat h_{\mathfrak b}) \overset{\text{\eqref{eq:IU.1}}}= \int_{E} p_t^\mathfrak{b}(x,y)h_\mathfrak{b}(x)^{-1} (C^1_{\mathfrak b,t,x,y})^{-1} m(dy), \quad x\in E,
		\\&\leq  h_\mathfrak{b}(x)^{-1} \left(\sup_{z\in E}(C^1_{\mathfrak b,t,x,z})^{-1}\right)  \int_{E} p_t^\mathfrak{b}(x,y)m(dy)
	\\& < \infty \quad\text{by \eqref{eq:IU.0} and \eqref{eq:IU.11} and the strict positivity of $h_\mathfrak{b}$}.
\end{align}
	This allows us to define a probability measure $\nu_\mathfrak b (dx):= m(\widehat h_{\mathfrak b})^{-1} \widehat h_\mathfrak b (x)m(dx), x\in E$, and an eigenfunction $\phi_\mathfrak{b}(x) := m(\widehat h_{\mathfrak b}) h_\mathfrak b(x), x\in E$.

Finally we write $\lambda := \lambda_\beta$ and assume that $ \lambda< 0$. We now show that $X$ satisfies \eqref{asp:H2!} with $\phi:=\phi_\beta$ and $\nu:= \nu_\beta$.
	From their definitions, we see that the function $\phi \in \mathcal B_b(E,(0,\infty))$, and that the probability measure $\nu$ has full support on $E$.
	Further, it is easy to see that for each $t\geq 0$, $P_t^\beta \phi = e^{\lambda t}\phi$ and $\nu(\phi) = 1$.
	We also have that for any $t>0$,
\begin{align}
	&(\nu P_t^\beta)(dy) = \int_{x\in E}p_{t}^\beta(x,y)m(dy) \nu(dx)
	\\&= \int_{x\in E}p_{t}^\beta(x,y)m(dy) m(\widehat h_\beta)^{-1}\widehat h_\beta(x)m(dx)
	\\&=  m(\widehat h_\beta)^{-1}  \left(\int_{x\in E} p_t^\beta(x,y) \widehat h_\beta(x) m(dx) \right) m(dy)
	\\& = m(\widehat h_\beta)^{-1} e^{\lambda t}\widehat h_\beta(y) m(dy) =
	e^{\lambda t}\nu(dy).
\end{align}
	Therefore $\nu P_t^\beta = e^{\lambda t}\nu, t\geq 0$. Now for each $t>0, x \in E$ and $f\in L_1^+(\nu)$, we have
\begin{align}
	&P_t^\beta f(x) = \int_{E} p^\beta_t(x,y) f(y)m(dy)
	\overset{\text{\eqref{eq:IU.1}}}= \int_{E} h_\beta (x) \widehat h_\beta (y) C^1_{\beta,t,x,y} f(y) m(dy)
	\\&= \int_{E} \phi (x)  C^1_{\beta,t,x,y} f(y) \nu(dy)
	=: e^{\lambda t} \phi(x) \nu(f) (1+ H_{t,x,f}).
\end{align}
	Finally, from \eqref{eq:IU.1} and \eqref{eq:IU.2}, it is elementary to verify that $H_{t,x,f}$ satisfies the required condition \eqref{asp:H2!}.

	In three paragraphs above, we give some conditions that imply \eqref{asp:H2!}. See \cite[Section 1.4]{RenSongZhang2015Limit} for more than 10 concrete examples of processes satisfying these conditions.

	\subsection*{Organization of the rest of the paper.}
	In Subsection \ref{subsec:OY} we will give the proof of Theorem \ref{Theorem:Y:H1:H2:H3:H4}
    using Propositions \ref{prop:Vf1::H1:H2::Y}--\ref{prop::GD:H1:H2:H3:H4::Y}.
	In Subsection \ref{subsec:QSD} we will give the proof of Theorem \ref{thm:QSD}
	using Propositions \ref{prop:EQ}--\ref{prop:UC}.
	The proofs of Propositions \ref{prop:Vf1::H1:H2::Y}--\ref{prop::GD:H1:H2:H3:H4::Y} are given in
	Section \ref{sec:propsforthm1}.
 The proof of  Propositions \ref{prop:EQ}--\ref{prop:UC} are given in Section \ref{sec:propsforthm2}.
	Some technical lemmas are in the Appendix, and will be referred to as needed in the proofs.

\section{Proofs of Theorems \ref{Theorem:Y:H1:H2:H3:H4} and \ref{thm:QSD}} \label{sec:proofof thms}

\subsection{Proof of Theorem \ref{Theorem:Y:H1:H2:H3:H4}} \label{subsec:OY}
	It is easy to see that the operators $(V_t)_{t\geq 0}$
	given by \eqref{eq:BGD.1} can be extended uniquely to a family of operators $(\overline V_t)_{t\geq 0}$ on $\mathcal B(E,[0,\infty])$ such that for all $t\geq 0$, $f_n \uparrow f$ pointwisely in  $\mathcal B(E, [0,\infty])$ implies that $\overline V_tf_n \uparrow \overline V_tf$ pointwisely.
	Moreover, $(\overline V_t)_{t\geq 0}$ satisfies that
\begin{equation}
\begin{minipage}{0.9\textwidth}
	$\overline V_t f \leq \overline V_t g$ for $t\geq 0$ and $f\leq g$ in $\mathcal B(E,[0,\infty])$;
\end{minipage}\label{Fact:BV!}
\end{equation}
\begin{equation}
\begin{minipage}{0.9\textwidth}
	$\overline V_{t+s} = \overline V_t \overline V_s$ for $t, s\geq 0$;  and
\end{minipage} \label{eq:OY.0}
\end{equation}
\begin{equation}
\begin{minipage}{0.9\textwidth}
	$\mathbb P_\mu [e^{-X_t(f)}] = e^{- \mu(\overline V_tf)}$ for $t\geq 0$, $\mu \in \mathcal M_f(E)$, and $f\in \mathcal B(E,[0,\infty])$.
\end{minipage} \label{eq:BGD.2}
\end{equation}
	With some abuse of notation, we still write $V_t = \overline V_t$ for $t\geq 0$, and call $(V_t)_{t\geq 0}$ \emph{the extended cumulant semigroup} of the superprocess $X$.
	Define $v_t = V_t(\infty  \mathbf 1_E)$ for $t\geq 0$, then it holds that
\begin{equation} \label{eq:OY.1}
	\mathbb P_\mu (\|X_t\| = 0)
	= e^{- \mu (v_t)},
	\quad \mu \in \mathcal M_f(E), t\geq 0.
\end{equation}
	From this, we can verify that
\begin{equation}\label{lem:sv2!}
	\text{$\mu(v_t) > 0$ for all $\mu \in \mathcal M_f^o(E)$ and $t \geq 0$.}
\end{equation}
	In fact, if $\mu(v_t) = 0$, then by \eqref{eq:OY.1} we have $\mathbb P_\mu(\|X_t \| = 0) = 1$, which contradicts \eqref{lem:Nd!}.

	In the  proof of Theorem \ref{Theorem:Y:H1:H2:H3:H4}, we will use the following four propositions whose proofs are postponed to Subsections \ref{sec:Vf1}, \ref{sec:Vf2}, \ref{sec:G} and \ref{sec:GD} respectively.

\begin{prop} \label{prop:Vf1::H1:H2::Y}
	For any $f\in \mathcal B(E, [0,\infty]),~t > T$ and $x\in E$, we have $V_tf(x) = C^3_{t,x,f} \phi(x)$ for some non-negative $C^3_{t,x,f}$ with $\lim_{t\to \infty} \sup_{x\in E, f\in \mathcal B(E, [0,\infty])}  C^3_{t,x,f} = 0$.
	In particular, we have $\lim_{t\to \infty} \mu(V_tf)= 0 $ for all $\mu \in \mathcal M_f(E)$ and $f\in \mathcal B(E,[0,\infty])$.
\end{prop}

\begin{prop} \label{prop:Vf2}
	For any $f\in \mathcal B(E,[0,\infty]),~t>T$ and $x\in E$, we have $V_tf(x) = \phi(x) \nu (V_tf) (1+C^4_{t,x,f}) $ for some real $C^4_{t,x,f}$ with $\lim_{t\to \infty} \sup_{x\in E, f\in \mathcal B(E, [0,\infty])} |C^4_{t,x,f}| = 0$.
\end{prop}

	For a probability measure $\mathbf P$ on $\mathcal M_f(E)$,
    the log-Laplace functional of $\mathbf P$ is defined by
	\[
	\mathscr L_\mathbf P f := - \log \int_{\mathcal M_f(E)}  e^{-\mu(f)} \mathbf P(d\mu), \quad
	f\in \mathcal B(E,[0,\infty]).
	\]
	For a finite random measure $\{Y; \mathbf P\}$, the log-Laplace functional of its distribution is denoted as $\mathscr L_{Y;\mathbf P}$.
	To simplify our notation, for each $t\geq 0$, we write
	$\Gamma_t := \mathscr L_{X_t;\mathbb P_\nu(\cdot | \|X_t\|>0)}$.
	
	We say a $[0,\infty]$-valued functional $A$ defined on $\mathcal B(E,[0,\infty])$ is monotone concave if
	(1) $A$ is a monotone functional, i.e., $f\leq g$ in $\mathcal B(E,[0,\infty])$ implies $Af \leq Ag$; and
	(2) for any $f\in \mathcal B(E,[0,\infty])$ with $Af< \infty$, the function $u \mapsto A(uf)$ is concave on $[0,1]$.

\begin{prop} \label{prop:G}
	The limit $Gf:= \lim_{t\to \infty} \Gamma_t f$ exists in $[0,\infty]$ for each $f\in \mathcal B(E,[0,\infty])$.
	Moreover, $G$ is the unique $[0,\infty]$-valued monotone concave functional on $\mathcal B(E,[0,\infty])$ such that
	$G(\infty  \mathbf 1_E) = \infty$ and that
\begin{equation} \label{eq:G.0}
	1 - e^{- GV_s f}
	= e^{s\lambda} (1 - e^{-Gf}),
	\quad s\geq 0, f\in \mathcal B(E,[0,\infty]).
\end{equation}
\end{prop}

\begin{prop} \label{prop::GD:H1:H2:H3:H4::Y}
	For any $g\in \mathcal B_b(E,[0,\infty))$ and
	sequence $(g_n)_{n\in \mathbb N}$ in $\mathcal B_b(E,[0,\infty))$
	such that $g_n \downarrow g$ pointwisely,
	we have $G g_n \downarrow Gg$.
\end{prop}

\begin{proof}[Proof of Theorem \ref{Theorem:Y:H1:H2:H3:H4}]
	It follows from Lemma \ref{fact:WC}, Propositions \ref{prop:G} and \ref{prop::GD:H1:H2:H3:H4::Y} that there exists a unique probability measure $\mathbf Q_\lambda$ on $\mathcal M_f(E)$ such that
\begin{equation}\label{eq:Y.0}
 	\mathbb P_{\nu}(X_t \in \cdot | \|X_t\|>0 )
	\xrightarrow[t\to \infty]{w} \mathbf Q_\lambda(\cdot)
\end{equation}
	and that
\begin{equation} \label{eq:Y.00}
	\mathscr L_{\mathbf Q_\lambda} = G \quad \mbox{on } \mathcal B_b(E,[0,\infty)).
\end{equation}

	We claim that \eqref{eq:Y.00} can be strengthened as
\begin{equation} \label{eq:Y.000}
	\mathscr L_{\mathbf Q_\lambda}
	= G \quad \mbox{on } \mathcal B(E,[0,\infty]);
\end{equation}
	and as a consequence of this, $\mathscr L_{\mathbf Q_\lambda}(\infty \mathbf 1_E) = G(\infty \mathbf 1_E)= \infty$, which says that $\mathbf Q_\lambda$ is actually a probability measure on $\mathcal M_f^o(E)$.
	To see the claim is true, we first note from Proposition \ref{prop:Vf1::H1:H2::Y} that
\begin{equation} \label{eq:EQ.1} \begin{minipage}{0.9\textwidth}
	there exists $T_1>0$ such that, for all $t>T_1$ and $f\in \mathcal B(E,[0,\infty])$, $V_tf \in \mathcal B_b(E,[0,\infty))$.
\end{minipage} \end{equation}
	 We then notice that from \eqref{eq:Y.00} and the bounded convergence theorem,
\begin{equation} \label{eq:Y.001}
\begin{minipage}{0.9\textwidth}
	if $\{g_n:n\in \mathbb N\} \cup \{g\} \subset \mathcal B_b(E,[0,\infty))$ and $g_n \uparrow g$ pointwisely, then $Gg_n \uparrow Gg$.
\end{minipage}
\end{equation}
	Now let $\{g_n:n\in \mathbb N\} \cup \{g\} \subset \mathcal B(E,[0,\infty])$ and $g_n \uparrow g$ pointwisely.
	Taking and fixing an $s > T_1$, we have by \eqref{eq:EQ.1} and \eqref{eq:Y.001} that
\[
	(1 - e^{-G g_n})
	\overset{\eqref{eq:G.0}} = e^{- s\lambda } (1 - e^{- G V_s g_n})
	\uparrow e^{- s\lambda } (1 - e^{- G V_s g})
	\overset{\eqref{eq:G.0}} =(1 - e^{-G g}).
\]
	In other word, we showed that $Gg_n \uparrow Gg$.
	The desired claim follows from this and \eqref{eq:Y.00}.

	Let us now prove that the probability $\mathbf Q_\lambda$ on $\mathcal M_f^o(E)$
	satisfies the requirement for the desired result.
	It follows from Proposition \ref{prop:Vf2} that there exists $T_2 >0$ such that
	$\sup_{x\in E, f \in \mathcal B(E,[0,\infty])} |C^4_{t,x,f}|< \infty$ for $t>T_2$.
	Thus for $f \in \mathcal B(E,[0,\infty])$, $t>T_2$ and $\mu \in \mathcal M_f^o(E)$, we have
\begin{align}
	\mu(V_tf)
		& \overset{\text{Proposition \ref{prop:Vf2}}}=
		\int_E  \phi(x) \nu (V_tf) (1+ C^4_{t,x,f})\mu(dx)
	\\ \label{eq:Y.1}
		& = \nu(V_tf) \mu(\phi)(1+ C^5_{\mu,t,f})
\end{align}
		for some real $C^{5}_{\mu,t,f}$ with $\lim_{t\to \infty} \sup_{f \in \mathcal B(E,[0,\infty])} |C^{5}_{\mu,t,f}| = 0$.
		Also note that for $f\in \mathcal B(E,[0,\infty])$,
	$t > T_2$ and $\mu \in \mathcal M_f^o(E)$,
\begin{align}
	&\mathbb P_\mu \left[1 - e^{-X_t(f)} \middle|\|X_t\|>0\right]
	\overset{\eqref{eq:BGD.2},\eqref{eq:OY.1}}= \frac{1 - e^{- \mu(V_tf)}} {1 - e^{-\mu(v_t)}}\\ \label{eq:Y.1.5}
    & = \frac{ \mu(V_t f) }{ \mu(v_t) }
	(1+C^{6}_{\mu,t,f})
\end{align}
	for some real $C^{6}_{\mu,t,f}$ with $\lim_{t\to \infty} |C^{6}_{\mu,t,f}| = 0$.
	Here in the last equality we used \eqref{lem:sv2!}, Proposition \ref{prop:Vf1::H1:H2::Y} and the fact that $(1-e^{-x})/x \xrightarrow[x\to 0]{}1$.
	Thus, for each $\mu \in \mathcal M^o_f(E)$ and $f\in C_b(E,[0,\infty))$, we have
\begin{align}
	&\mathbb P_\mu \left[1 - e^{-X_t(f)} \middle|\|X_t\|>0\right]
	\overset{\text{\eqref{eq:Y.1}, \eqref{eq:Y.1.5}}}= \frac{ \nu(V_tf) }{ \nu(v_t) }
	\frac{1+C^5_{\mu,t,f}}{1+C^5_{\mu, t,\infty \mathbf 1_E}}(1+ C^6_{\mu,t,f})
	\\& \overset{\text{\eqref{eq:Y.1.5}}}= \mathbb P_\nu \left[1 - e^{-X_t(f)} \middle| \|X_t\|>0\right]
(1+C^6_{\nu, t,f})^{-1}  \frac{1+C^5_{\mu,t,f}}{1+C^5_{\mu,  t,\infty \mathbf 1_E}}(1+ C^6_{\mu,t,f})
	\\&\xrightarrow[t\to \infty]{} \int_{\mathcal M_f(E)}(1-e^{-w(f)}) \mathbf Q_\lambda(dw),
\end{align}
	where in the last line above, we used \eqref{eq:Y.0}.
	Therefore, according to \cite[Theorem 1.18]{Li2011MeasureValued},
	\[\mathbb P_\mu\left(X_t \in \cdot \middle| \|X_t\|>0\right)
	\xrightarrow[t\to \infty]{w}
	\mathbf Q_\lambda(\cdot). \qedhere\]
\end{proof}

\subsection{Proof of Theorem \ref{thm:QSD}} \label{subsec:QSD}
	In this subsection,
	we give the proof of Theorem \ref{thm:QSD} using the following three
Propositions \ref{prop:EQ}, \ref{prop:CQ} and \ref{prop:UC}  whose proofs are postponed to
Subsection \ref{sec:EQ}, \ref{sec:CQ} and \ref{sec:UC}, respectively.

\begin{prop} \label{prop:EQ}
	(1) The Yaglom limit $\mathbf Q_\lambda$ given by Theorem \ref{Theorem:Y:H1:H2:H3:H4} is a QSD of $X$ with mass decay rate $\lambda$; and
	(2) for any $r \in (\lambda , 0)$,
	there exists a probability measure $\mathbf Q_r$  on $\mathcal M^o_f(E)$
such that  $\mathbf Q_r$ is a QSD of $X$ with mass decay rate $r$.
\end{prop}

\begin{prop} \label{prop:CQ}
	Suppose that $r \in (-\infty, 0)$ and that  $\mathbf Q^*_{r}$
	is a QSD
	for $X$ with mass decay rate $r$.
	Then we have that (1) $r \geq \lambda$; and
	(2) $\mathscr L_{\mathbf Q^*_r}$ is a monotone concave functional on $\mathcal B(E,[0,\infty])$
	with $\mathscr L_{\mathbf Q^*_r}(\infty \mathbf 1_E) = \infty$ and that
\[
	1 - e^{- \mathscr L_{\mathbf Q^*_r} V_s f}
	= e^{sr}(1- e^{- \mathscr L_{\mathbf Q^*_r} f}),
	\quad s\geq 0, f\in \mathcal B(E,[0,\infty]).
\]
\end{prop}

\begin{prop} \label{prop:UC}
	Let $G$ be the unique functional on $\mathcal B(E,[0,\infty])$ given by Proposition \ref{prop:G}.
	Let $r \in [\lambda, 0)$.
	If $G_r$ is a monotone concave functional on $\mathcal B(E,[0,\infty])$ with
	$G_r(\infty \mathbf 1_E) = \infty$ and that
\[
	1 - e^{-G_r V_s f}
	= e^{sr }(1- e^{- G_r f}),
	\quad s\geq 0, f\in \mathcal B(E,[0,\infty]),
\]
then $1 - e^{-G_rf} = (1 - e^{- G f})^{r/\lambda}$ for any $f\in \mathcal B(E,[0,\infty])$.
\end{prop}

\begin{proof}[Proof of Theorem \ref{thm:QSD}]
	The non-existence of QSD for $X$ with mass decay rate $r < \lambda$ is due to Proposition \ref{prop:CQ} (1).
	The existence of QSD for $X$ with mass decay rate $r \in [\lambda,0)$ is due to Proposition \ref{prop:EQ}.
	The uniqueness of QSD for $X$ with mass decay rate $r\in [\lambda , 0)$ is due to Propositions \ref{prop:CQ}, \ref{prop:UC} and \cite[Theorem 1.17]{Li2011MeasureValued}.
\end{proof}

\section{Proofs of Propositions \ref{prop:Vf1::H1:H2::Y}--\ref{prop::GD:H1:H2:H3:H4::Y}} \label{sec:propsforthm1}
\subsection{Proof of Proposition \ref{prop:Vf1::H1:H2::Y}} \label{sec:Vf1}
	Define a function $\psi_0$ by
\[
	\psi_0(x,z) = \psi(x,z)+ \beta(x) z, \quad x\in E, z\in [0,\infty),
\]	
	and an operator $\Psi_0: \mathcal B(E, [0,\infty]) \to \mathcal B(E,[0,\infty])$ by
\begin{equation}
	\Psi_0 f(x)
	= \lim_{n\to \infty} \psi_0(x,f(x) \wedge n),
	\quad f\in \mathcal B(E,[0,\infty]), x\in E.
\end{equation}
	Then it follows from \cite[Theorem 2.23]{Li2011MeasureValued} and monotonicity that
\begin{equation}\label{eq:Vf1.1}
	V_s f + \int_0^s P_{s-u}^\beta \Psi_0 V_{u} f ~du
	= P_s^\beta f,
	\quad f\in \mathcal B(E,[0,\infty]), s\geq 0.
\end{equation}

	The following fact will be used repeatedly:
\begin{equation} \label{lem:nV::H2::Vf1}
	\{V_tf:t> T, f\in \mathcal B(E, [0,\infty])\}\subset L_1^+(\nu).
\end{equation}
	To see this, note from \eqref{Fact:BV!}, \eqref{eq:OY.1} and \eqref{asp:H4!} that, for  all $t> T$ and $f\in \mathcal B(E,[0,\infty])$, $\nu(V_t f) \leq \nu(v_t)   = - \log \mathbb P_\nu (\|X_t\| = 0)  < \infty. $

\begin{proof}[{Proof of Proposition \ref{prop:Vf1::H1:H2::Y}}]
	Note that for all $s>0$ and $\epsilon>0$,
\begin{align}
	& V_{s+\epsilon +T} f (x)
	\overset{\eqref{eq:OY.0}}= V_s V_{T+\epsilon} f(x)
	\leq P_s^\beta V_{T + \epsilon} f(x)\quad\text{by \eqref{eq:Vf1.1}},
 	\\ \label{eq:Vf1.2} & \overset{\eqref{asp:H2!},\eqref{lem:nV::H2::Vf1}}= e^{\lambda s}\phi(x) \nu( V_{T +\epsilon} f)  (1+ H_{s,x,V_{T + \epsilon} f})
	\\&\leq e^{\lambda s}\phi(x) \nu(v_{T +\epsilon})  (1+ \sup_{x\in E, g\in L_1^+(\nu)}|H_{s,x,g}|),
\end{align}
	where
	in the last inequality we used the fact that $\nu(V_t f) \leq \nu(v_t)   = - \log \mathbb P_\nu (\|X_t\| = 0)  < \infty$
	for all  $f\in \mathcal B(E,[0,\infty])$ and $t > T$.
	From this and the fact that $\lambda < 0$, we immediately get the desired result.
\end{proof}

\subsection{Proof of Proposition \ref{prop:Vf2}} \label{sec:Vf2}
	Another fact that will be used repeatedly is the following:
\begin{equation}
\begin{minipage}{0.9\textwidth}
	For any $f\in \mathcal B(E,[0,\infty])$, $\nu(f) = 0$ implies $\nu(V_tf)=0$ for all $t\ge 0$; and $\nu(f)>0$ implies $\nu(V_tf)>0$ for all $t\ge 0$.
\end{minipage}\label{lem:nVn!}
\end{equation}
	To see this, note by \eqref{Fact:M!} that $ \mathbb P_\nu[X_t(f)] = \nu (P_t^\beta f) = e^{\lambda t}\nu (f). $
	If $\nu(f) = 0$, then $X_t(f)=0, \mathbb P_\nu$-a.s., therefore $\nu(V_t f) = - \log \mathbb P_\nu[e^{-X_t(f)}] =0. $
	If $\nu(f) > 0$, then under $\mathbb P_\nu$, $X_t(f)$ is a random variable with positive mean.
	Therefore, $ \nu(V_tf) = - \log \mathbb P_\nu[e^{-X_t(f)}] >0$.

	Combining \eqref{lem:nVn!} with \eqref{eq:Vf1.2} we get that
\begin{equation}
\begin{minipage}{0.9\textwidth}
	for all $t>T,~x\in E$ and $f \in \mathcal B(E,[0,\infty])$ with $\nu(f) = 0$, we have $V_t f(x ) = 0$.
\end{minipage} \label{lem:nullVf}
\end{equation}
Note  from \eqref{asp:H2!} and \eqref{lem:nV::H2::Vf1} that for all $s>0, t> T, x\in E$ and $f\in \mathcal B(E,[0,\infty])$, we have
\begin{equation}
	P_s^\beta V_tf(x)  =e^{\lambda s} \phi(x)\nu(V_tf) (1+H_{s,x,V_tf}) <\infty.
\label{lem:PV}
\end{equation}

	In the proof of Proposition \ref{prop:Vf2} we will use the following three lemmas whose proofs are postponed later.

\begin{lem} \label{prop:PVf}
 For all $s> 0,~t> T,~ x\in E$ and $f\in \mathcal B(E,[0,\infty])$, we have $P_s^\beta V_t f(x) = \phi(x) \nu(V_{t+s}f) (1+C^7_{s,t,x,f})$ for some real $C^7_{s,t,x,f}$ with
\[
	\lim_{s\to \infty} \varlimsup_{t\to \infty} \sup_{x\in E, f\in \mathcal B(E,[0,\infty])} |C^7_{s,t,x,f}|
		= 0.
\]
\end{lem}

For $f\in \mathcal B(E,[0,\infty])$ and $0 < \epsilon < s < \infty$, we define
\begin{equation}
	I_{s,\epsilon} f
 	= \int_0^{s - \epsilon} P_{s - u}^\beta \Psi_0 V_u f ~du, \qquad
 	J_{s,\epsilon} f
 	= \int_{s-\epsilon}^s P_{s-u}^\beta \Psi_0 V_u f ~du.
\end{equation}
\begin{lem} \label{prop:IVf}
		For all $t> T,~0<\epsilon<s< \infty,~x\in E$ and $f\in \mathcal B(E,[0,\infty])$ with $\nu(f)>0$, we have $I_{s,\epsilon}V_t f(x) = \phi(x) \nu(V_{s+t} f) C^8_{t,\epsilon, s, x,f}$ for some non-negative $C^8_{t,\epsilon, s, x,f}$ with
\[
	\lim_{t\to \infty} \sup_{x\in E, f\in \mathcal B(E,[0,\infty])} C^8_{t,\epsilon, s, x,f} = 0.
\]
\end{lem}

\begin{lem} \label{prop:JVf}
For all $t>T,~0<\epsilon<s< \infty,~x\in E$ and $f\in \mathcal B(E,[0,\infty])$ with $\nu(f)>0$, we have $ J_{s,\epsilon} V_tf(x) = \phi(x) \nu(V_{s+t}f) C^9_{t,\epsilon,s,x,f}$ for some non-negative $C^9_{t,\epsilon,s,x,f}$ with
\[
	\lim_{\epsilon \to 0}\varlimsup_{t+s\to \infty} \sup_{x\in E, f\in \mathcal B(E,[0,\infty])} C^9_{t,\epsilon,s,x,f} =0.
\]
\end{lem}

\begin{proof}[{Proof of Proposition \ref{prop:Vf2}}]
Thanks to \eqref{lem:nVn!} and \eqref{lem:nullVf}, we only need to consider the case that $\nu(f)>0$.
	In this case,
by Lemmas \ref{prop:PVf}, \ref{prop:IVf} and \ref{prop:JVf}.
we have for any $s>0$ and $\epsilon\in (0,s)$,
\begin{align}
	& V_{t+s}f (x)
	\overset{\eqref{eq:OY.0}}= V_s V_t f(x)
	\overset{\eqref{eq:Vf1.1},\eqref{lem:PV}}= P_s^\beta V_t f(x) - \int_0^s P^\beta_{s-u}\Psi_0 V_uV_t f(x) du
	\\&= P_s^\beta V_t f(x) - I_{s,\epsilon} V_tf(x) - J_{s,\epsilon} V_t f(x)
		\\\label{eq:Vf2.1} &=\phi(x)\nu(V_{t+s}f)
		\left( 1+ C^7_{s,t,x,f} - C^8_{t,\epsilon, s, x,f}- C^9_{t,\epsilon,s,x,f}\right).
\end{align}
	On the other hand, we have
\begin{equation}\label{eq:Vf2.2}
	V_{t}f (x)
		= \phi(x) \nu(V_{t}f) (1+ C^{10}_{t,x,f})		\quad \text{for some real $C^{10}_{t,x,f}$}.
\end{equation}
	Combining \eqref{eq:Vf2.1} and \eqref{eq:Vf2.2}, we have for all $s>0$ and $\epsilon \in (0,s)$,
\[
	C^{10}_{t+s,x,f} = C^{7}_{s,t,x,f} - C^{8}_{t,\epsilon, s, x,f}- C^{9}_{t,\epsilon,s,x,f}.
\]
	 Using this and the fact that
\[
	\lim_{\epsilon \to 0}\varlimsup_{s\to \infty}\varlimsup_{t\to \infty}
	\sup_{x\in E, f\in \mathcal B(E,[0,\infty])}
	|C^7_{s,t,x,f} - C^8_{t,\epsilon, s, x,f}- C^9_{t,\epsilon,s,x,f}|=0,
\]
	it is easy to check that $\varlimsup_{s\to \infty}\varlimsup_{t\to \infty} \sup_{x\in E, f\in
	\mathcal B(E,[0,\infty])} |C^{10}_{t+s,x,f}|=0$.
	This implies
\[
	\lim_{t\to \infty}
	\sup_{x\in E, f\in \mathcal B(E,[0,\infty])}
	|C^{10}_{t,x,f}|=0.
	\qedhere
\]
\end{proof}

	Now we prove the three lemmas above.

\begin{proof}[Proof of Lemma \ref{prop:PVf}]
	
	Integrating  both sides of \eqref{eq:Vf1.1} with respect to $\nu$ and replacing $f$ by $V_t f$, we get that for all $t, s\geq 0$ and $f\in \mathcal B(E,[0,\infty])$,
\begin{equation}\label{eq:nuP.1}
	e^{- \lambda (t+s)} \nu(V_{t+s}f) + \int_0^s e^{- \lambda (t+u)} \nu(\Psi_0 V_{t+u}f) du
	= e^{- \lambda t} \nu(V_t f).
\end{equation}
	As a consequence of \eqref{eq:nuP.1}, we can get that for all $t>T$, $s\geq 0$ and $f \in \mathcal B(E,[0,\infty])$ with $\nu(f)>0$,
\begin{equation}\label{Claim:nVI}
	\frac{\nu(V_{t+s} f)} {\nu(V_t f)}
	= \exp\left\{ \lambda s - \int_t^{t+s} \frac{\nu(\Psi_0 V_u f) }{\nu(V_u f)} ~du\right\}.
\end{equation}
	In fact, first observe from \eqref{lem:nV::H2::Vf1} and \eqref{lem:nVn!} that both sides of \eqref{eq:nuP.1} are finite and positive if $t> T$ and $\nu(f)>0$.
	Therefore the function $H: u\mapsto e^{-\lambda u}\nu(V_u f)$ is absolutely continuous on $(T,\infty)$ and
\begin{equation}
	d H(u)
	= - e^{- \lambda u} \nu(\Psi_0 V_u f) du,
	\quad u\in (T,\infty),
\end{equation}
	which implies that
\begin{equation}
	d \log H(u)
	= - \frac{\nu(\Psi_0 V_u f )}{ \nu(V_u f)} du,
	\quad u \in (T,\infty).
\end{equation}
	Now an elementary integration argument gives \eqref{Claim:nVI}.
	
	Define an operator $\Psi_0'$ on $\mathcal B(E,[0,\infty])$ by
\begin{equation}
	\Psi_0' f(x)
	= \lim_{n\to \infty}\frac{\partial \psi_0}{ \partial z} (x, n\wedge f(x)),
	\quad x\in E, f\in \mathcal B(E,[0,\infty]).
\end{equation}
	We first claim that for all $t> T, x\in E$ and $f\in \mathcal B(E,[0,\infty])$,
\begin{equation}\label{lem:PsV}
	\varlimsup_{t\to \infty} \sup_{x\in E, f\in \mathcal B(E, [0,\infty])}\Psi_0'V_tf(x) < \infty.
\end{equation}
	In fact, since
\begin{equation}\label{e:derofpsi0-2}
	\frac{\partial \psi_0 }{ \partial z} (x,z)
	= 2\sigma (x)^2 z + \int_0^\infty (1 - e^{- rz}) r \pi(x,dr),
	\quad x\in E, z\geq 0,
\end{equation}
	we have,
\begin{align}
	&\Psi_0' V_tf(x)
	\leq 2\sigma (x)^2 V_t f(x) + V_t f(x) \int_0^1 r^2 \pi(x,dr) + \int_1^\infty r \pi(x,dr)
		\\&\overset{\text{Proposition \ref{prop:Vf1::H1:H2::Y}}}
	= C^3_{t,x,f}
	\phi(x) \left(2\sigma (x)^2 +\int_0^1 r^2 \pi(x,dr) \right)+ \int_1^\infty r \pi(x,dr),
\end{align}
	Since $\phi$, $\sigma$ are bounded,  and $(r\wedge r^2)\pi(x,du)$ is a bounded kernel, \eqref{lem:PsV} follows easily.
	
	We next claim that for all $t > T$ and $f\in \mathcal B(E,[0,\infty])$,
\begin{equation}\label{eq:nPPV}
	\lim_{t\to \infty}\sup_{f\in \mathcal B(E, [0,\infty])} \nu(\Psi_0' V_t f) = 0.
\end{equation}
	In fact, it follows from \eqref{e:derofpsi0-2} that, for any fixed $x\in E$, $z\mapsto \frac{\partial \psi_0}{\partial z} (x,z)$ is a non-negative, non-decreasing and continuous function on $[0,\infty)$ with $\frac{\partial \psi_0}{\partial z} (\cdot,0) \equiv 0$.
	Therefore for any $x\in E$, we have
\[
	\lim_{t\to \infty} \Psi_0' v_t(x) =\lim_{t\to \infty} \frac{\partial \psi_0}{ \partial z}(x,v_t(x)) \overset{\text{Proposition \ref{prop:Vf1::H1:H2::Y}}}{=} 0.
\]
	Using this, \eqref{lem:PsV} and the bounded convergence theorem, we easily get $\lim_{t\to \infty}\nu(\Psi_0' v_t)  = 0. $ The claim follows immediately from the monotonicity of $\Psi_0' V_t f$ in $f\in \mathcal B(E,[0,\infty])$.

	Here is another claim that will be used below:
\begin{equation}
\begin{minipage}{0.9\textwidth}
	For all $t> T, x\in E$ and $f\in \mathcal B(E,[0,\infty])$, it holds that \[V_t f(x) = \phi(x) \nu(V_tf) C^{11}_{t,x,f}\] for some non-negative $C^{11}_{t,x,f}$ with $\varlimsup_{t\to \infty} \sup_{x\in E, f\in \mathcal B(E,[0,\infty])} C^{11}_{t,x,f} <\infty$.
\end{minipage} \label{eq:VfO}
\end{equation}
	To see this, first note that \eqref{eq:VfO} is trivial when $\nu(f) = 0$ thanks to \eqref{lem:nVn!} and \eqref{lem:nullVf}.
	Therefore, we only need to consider the case that $\nu(f)>0$.
	In this case, it follows from the elementary fact
\begin{equation}\label{e:derofpsi0}
	\psi_0(x,z)
	\leq z \frac{\partial \psi_0}{\partial z}(x,z),
	\quad x\in E, z\geq 0,
\end{equation}
	that
\begin{align}
	&\nu(\Psi_0 V_tf)
    \leq \nu((V_tf)\cdot (\Psi_0' V_tf)) \leq \nu(V_tf) \sup_{y\in E} \Psi_0' V_tf(y).
	\end{align}
	From \eqref{lem:nV::H2::Vf1}  we get that $\nu(V_tf) <\infty$.
	Thus from \eqref{lem:PsV} for $t> T$ and $f\in \mathcal B(E,[0,\infty])$,
\begin{equation}
	\nu(\Psi_0 V_t f)
	= \nu(V_tf) C^{12}_{t,f} \label{eq:VfO.1}
\end{equation}
	for some non-negative $C^{12}_{t,f} $ with $\varlimsup_{t\to \infty} \sup_{f\in \mathcal B(E,[0,\infty])}C^{12}_{t,f}  < \infty$.
	Therefore, for any $s\geq 0$,
\begin{align}
	&  \frac{\nu(V_{t+s} f)} {\nu(V_t f)} \overset{\eqref{Claim:nVI}}= \exp\left\{ \lambda s - \int_t^{t+s} \frac{\nu(\Psi_0 V_u f) }{\nu(V_u f)} du\right\}
	\\&\label{eq:VfO.2} \overset{\text{\eqref{eq:VfO.1}}}= \exp\left\{ \lambda s -
	\int_t^{t+s} C^{12}_{u,f} ~du\right\}.
\end{align}
Now note that for any $\epsilon\in (0, t- T)$,
\begin{align}
	&V_{t}f(x) \overset{\eqref{Fact:BV!}}= V_{\epsilon} V_{t-\epsilon} f\leq P_\epsilon^\beta V_{t-\epsilon} f(x)\quad\text{by \eqref{eq:Vf1.1}},
	\\& \overset{\text{\eqref{asp:H2!}}}= \phi(x) \nu(V_{t-\epsilon}f) e^{\lambda  \epsilon} (1+H_{\epsilon,x, V_{t-\epsilon} f} )
		\\& \label{eq:VfO.3}\overset{\text{\eqref{eq:VfO.2}}}= \phi(x)\nu(V_{t}f)\exp\left\{ \int_{t-\epsilon}^t
	C^{12}_{u,f} ~du\right\} (1+H_{\epsilon,x, V_{t-\epsilon} f} ).
	\end{align}
	According to \eqref{lem:nV::H2::Vf1}  and \eqref{asp:H2!} we have
\begin{equation}	
	\varlimsup_{t\to \infty}\sup_{x\in E, f\in \mathcal B(E,[0,\infty])}
	|H_{\epsilon,x, V_{t-\epsilon} f}| < \infty, \quad \epsilon > 0.
\end{equation}
	From this, $\eqref{eq:VfO.3}$ and the fact that
	$\varlimsup_{u\to \infty} \sup_{f\in \mathcal B(E,[0,\infty])} C^{12}_{u,f}  < \infty$,
 \eqref{eq:VfO} follows immediately.
	
	We now use \eqref{lem:PsV}, \eqref{eq:nPPV} and \eqref{eq:VfO} to give the asymptotic ratio of $\nu(\Psi_0V_tf)$ and $\nu(V_tf)$.
	Note that we already obtained some result for this ratio in \eqref{eq:VfO.1}.
	We claim that the following stronger assertion is valid:
\begin{equation}\label{eq:nP}
	\lim_{t\to \infty}\sup_{f\in \mathcal B(E,[0,\infty])} C^{12}_{t,f} = 0,
	\quad f\in \mathcal B(E,[0,\infty]).
\end{equation}
	To see this, we observe that
\begin{align}
	&\nu(\Psi_0 V_tf)
	\leq \nu((V_tf)\cdot (\Psi_0' V_tf)) ,\quad\text{by \eqref{e:derofpsi0}},
	\\&  \leq  \nu(\Psi_0' V_tf) \sup_{x\in E}V_tf(x)
	\overset{\eqref{eq:VfO}}=   \nu(\Psi_0' V_tf)  \cdot \nu(V_tf) \sup_{x\in E}
	(\phi(x) C^{11}_{t,x,f}).
\end{align}
	Since $\phi$ is bounded, \eqref{eq:nP} follows from \eqref{eq:nPPV} and \eqref{eq:VfO}.

	Using \eqref{eq:nP}, we can get the following asymptotic ratio of $\nu(V_{t+s}f)$ and $\nu(V_tf)$:
\begin{equation} \label{eq:nVR}
\begin{minipage}{0.9\textwidth}
	For all $t> T,~s \geq 0$ and $f\in \mathcal B(E,[0,\infty])$, we have
\[
	\nu(V_{t+s}f) = \nu(V_tf) \exp\{\lambda s (1+C^{13}_{t,s,f}) \}
\]
	for some real $C^{13}_{t,s,f}$ with $\lim_{t\to \infty} \sup_{s\geq  0,f\in \mathcal B(E,[0,\infty])} |C^{13}_{t,s,f}| = 0$.
	In particular, for all $f\in \mathcal B(E,[0,\infty])$ with $\nu(f)>0$ and $s\geq 0$, we have $\lim_{t\to \infty} \frac{\nu(V_{t+s}f)}{\nu(V_tf)} = e^{\lambda s}$.
\end{minipage}
\end{equation}
	To see this, thanks to \eqref{lem:nVn!},
	we only need to consider the case $\nu(f)>0$. In this case,  it holds  that
\begin{align}
	&\frac{\nu(V_{t+s} f)} {\nu(V_t f)}
	\overset{\eqref{eq:VfO.2}}= \exp\left\{\lambda s- \int_t^{t+s} C^{12}_{u,f} ~du\right\}
		=: \exp\{\lambda s (1+C^{13}_{t,s,f}) \}.
	\end{align}
Noticing that $C^{13}_{t,s,f} = -\frac{1}{\lambda s}\int_t^{t+s} C^{12}_{u,f} ~du$ and by \eqref{eq:nP} that $\lim_{u\to \infty}\sup_{f\in \mathcal B(E,[0,\infty])} C^{12}_{u,f} = 0$,
so we have $\lim_{t\to \infty} \sup_{s> 0,f\in \mathcal B(E,[0,\infty])} |C^{13}_{t,s,f}| = 0$.	
	
	We are now  ready to prove the conclusion of Lemma \ref{prop:PVf}.
	Again we only need to consider the case $\nu(f)>0$ thanks to \eqref{lem:nVn!} and  \eqref{lem:nullVf}.
	In this case, by \eqref{lem:nV::H2::Vf1} and \eqref{lem:nVn!}, we have $0<\nu(V_{t}f)<\infty$.
	Therefore, we have
\begin{align}
	& P_s^\beta V_t f(x)
	\overset{\text{\eqref{asp:H2!}}}= e^{\lambda s} \phi(x) \nu(V_tf) (1+H_{s,x,V_tf})
	\\&\overset{\eqref{eq:nVR}}= \phi(x) \nu(V_{t+s}f)
	\exp\{-\lambda s C^{13}_{t,s,f}\} (1+H_{s,x,V_tf}).
	\end{align}
	From \eqref{asp:H2!} and \eqref{lem:nV::H2::Vf1}, we know that $\lim_{s\to \infty} \sup_{x\in E, t> T, f\in \mathcal B(E,[0,\infty])}|H_{s,x,V_tf}| = 0$.
From \eqref{eq:nVR}, we know that $\sup_{s\geq 0} \lim_{t\to \infty}
\sup_{f\in \mathcal B(E,[0,\infty])} |sC^{13}_{t,s,f}| = 0$.
	Therefore, we have
 \[
\lim_{s\to \infty}\varlimsup_{t\to \infty}\sup_{x\in E, f\in \mathcal B(E,[0,\infty])}
    |\exp\{-\lambda s C^{13}_{t,s,f}\} (1+H_{s,x,V_tf})-1| = 0.
 \]
Combining the displays above we get the conclusion of Lemma \ref{prop:PVf}.
\end{proof}

\begin{proof}[Proof of Lemma \ref{prop:IVf}]
	For all $u\geq 0$, we have
\begin{align} \label{eq:IVf.25}
	&\nu(P_u^\beta \Psi_0 V_t f) = e^{\lambda u}\nu(\Psi_0 V_t f)<\infty,
\end{align}
where the inequality follows from \eqref{eq:VfO.1} and \eqref{lem:nV::H2::Vf1}
	Therefore, we have
\begin{align}
 	& I_{s,\epsilon} V_t f(x)
 	= \int_0^{s- \epsilon} P_{s-u}^\beta \Psi_0 V_{t+u} f (x) du
 	= \int_0^{s- \epsilon} P_\epsilon^\beta (P_{s - \epsilon - u}^\beta \Psi_0 V_{t+u} f )(x) du
 	\\&\overset{\eqref{asp:H2!}}
 	= \int_0^{s - \epsilon} e^{\lambda \epsilon} \phi(x) \nu(P_{s - \epsilon - u}^{\beta} \Psi_0 V_{t+u} f)  \left(1+H_{\epsilon ,x , P_{s - \epsilon - u}^{\beta} \Psi_0 V_{t+u} f}\right) du
	\\&\overset{\text{\eqref{eq:IVf.25}}}= e^{(t+s)\lambda} \int_0^{s - \epsilon} \phi(x) e^{-\lambda (t+u)}\nu(\Psi_0 V_{t+u} f)  \Big(1+H_{\epsilon ,x , P_{s - \epsilon - u}^{\beta} \Psi_0 V_{t+u} f}\Big) du
	\\&\leq \phi(x) \Big(1+\sup_{g\in L_1^+(\nu)}|H_{\epsilon ,x , g}|\Big) e^{(t+s)\lambda} \int_0^{s} e^{-\lambda (t+u)} \nu(\Psi_0 V_{t+u}f)du\quad
	\text{by \eqref{eq:IVf.25}}
 	\\&\overset{\eqref{eq:nuP.1}}= \phi(x) \left(1+\sup_{g\in L_1^+(\nu)}|H_{\epsilon ,x , g}|\right)  e^{(t+s)\lambda} \Big(e^{-\lambda t}\nu(V_tf)- e^{-\lambda(t+s)}\nu(V_{t+s}f)\Big)
 	\\&\overset{\eqref{lem:nV::H2::Vf1},\eqref{lem:nVn!}}= \phi(x) \Big(1+\sup_{g\in L_1^+(\nu)}|H_{\epsilon ,x , g}|\Big) \nu(V_{t+s}f) \Big( \frac{e^{s \lambda }\nu(V_tf)}{\nu(V_{t+s}f)} - 1\Big)
  \\&\overset{\eqref{eq:nVR}}= \phi(x) \Big(1+\sup_{g\in L_1^+(\nu)}|H_{\epsilon ,x , g}|\Big)
  \nu(V_{t+s}f) ( \exp\{- \lambda s C^{13}_{t,s,f}\} - 1).
\end{align}
	It is easy to check that
\begin{equation}
	\lim_{t\to \infty}\sup_{x\in E, f\in \mathcal B(E,[0,\infty])}
	\Big|\Big(1+\sup_{g\in L_1^+(\nu)}|H_{\epsilon ,x , g}|\Big)
	( \exp\{- \lambda s C^{13}_{t,s,f}\} - 1) \Big| = 0.
\end{equation}
	The desired result then follows.
\end{proof}

\begin{proof}[Proof of Lemma \ref{prop:JVf}]
	It follows from \eqref{e:derofpsi0} that for all $t> T, x\in E$ and $f\in \mathcal B(E,[0,\infty])$,
\begin{equation}
	\Psi_0 V_t f (x)
	\leq V_tf(x)\cdot \Psi_0' V_t f(x)
	\end{equation}
		Now by \eqref{lem:PsV} we have
	\begin{equation} \label{lem:PVtV}
		\Psi_0 V_t f(x) = V_tf(x) C^{14}_{t,x,f}
\end{equation}
for some non-negative $C^{14}_{t,x,f}$ with $\varlimsup_{t\to \infty} \sup_{x\in E, f\in \mathcal B(E,[0,\infty])} C^{14}_{t,x,f} < \infty$.

Recall the quantity $C^{13}_{t,s,f}$ given in  \eqref{eq:nVR}.
	Now we claim that for all $u\geq 0$, $t>T$, $x\in E$ and $f\in \mathcal B(E,[0,\infty])$,
\begin{equation} \label{Claim:PuPVt}
        P_u^\beta \Psi_0 V_{t} f(x) = \phi(x)\nu(V_{t+u}f)
       \exp\{-\lambda u C^{13}_{t,u,f} \} C^{15}_{t,u,x,f}
\end{equation}
for some non-negative $C^{15}_{t,u,x,f}$ with $\varlimsup_{t\to \infty} \sup_{u\geq 0, x\in E, f\in \mathcal B(E,[0,\infty])} C^{15}_{t,u,x,f} < \infty$.
\begin{align}
	&P_u^\beta \Psi_0 V_{t} f(x)
	= \int_{E} \Psi_0V_tf(y) P_u^\beta (x,dy)
	\overset{\eqref{lem:PVtV}}=\int_{E} V_tf(y)
	C^{14}_{t,y,f} P_u^\beta (x,dy)
        \\&\overset{\eqref{eq:VfO}}=\int_{E} \phi(y)\nu(V_tf)
        C^{11}_{t,y,f}C^{14}_{t,y,f} P_u^\beta (x,dy)
\\&\overset{\eqref{eq:nVR}}=\int_{E} \phi(y)\nu(V_{t+u}f) \exp\{-\lambda u
(1+C^{13}_{t,u,f}) \} C^{11}_{t,y,f}C^{14}_{t,y,f} P_u^\beta (x,dy)
\\& \leq \nu(V_{t+u}f)
\exp\{-\lambda u (1+C^{13}_{t,u,f}) \} \Big(\sup_{z\in E} C^{11}_{t,z,f}C^{14}_{t,z,f}\Big)
\int_{E} \phi(y) P_u^\beta (x,dy)
\\& = \nu(V_{t+u}f)
\exp\{-\lambda u (1+C^{13}_{t,u,f}) \} \Big(\sup_{z\in E} C^{11}_{t,z,f}C^{14}_{t,z,f}\Big)
e^{\lambda u}\phi(x).
\end{align}
 Now \eqref{Claim:PuPVt} follows from the fact that $\varlimsup_{t\to \infty}
 \Big(\sup_{z\in E, f\in \mathcal B(E,[0,\infty])} C^{11}_{t,z,f} C^{14}_{t,z,f}\Big) < \infty$.

	Note that \eqref{Claim:PuPVt} gives the asymptotic behavior of $P_u^\beta \Psi_0 V_t f(x)$.
	We want to reformulate it into the asymptotic behavior  of $P_u^\beta \Psi_0 V_{t-u} f(x)$.	
	To do this, we use the following elementary facts: for any real function $h$ on $[0,\infty)^2$,
\begin{align}\label{Fact:TO!}
	\varlimsup_{t\to \infty} \sup_{u\geq 0} |h(t,u)| < \infty & \implies \sup_{\epsilon > 0} \varlimsup_{t\to \infty} \sup_{u \in (0,\epsilon)} |h(t-u,u)| < \infty;
	\\ 	\lim_{t\to \infty} \sup_{u \geq 0} |h(t,u)| = 0 & \implies \sup_{\epsilon > 0} \lim_{t\to \infty} \sup_{u \in (0,\epsilon)} u\cdot |h(t-u,u)| = 0.
\end{align}
	Observe that for all $u>0$, $t> T + u$ and $f \in \mathcal B(E,[0,\infty])$,
\begin{align}
	& P_u^\beta \Psi_0 V_{t-u} f(x)
	\overset{\eqref{Claim:PuPVt}}= \phi(x) \nu(V_{t}f)
	\exp\{-\lambda u C^{13}_{t-u,u,f} \} C^{15}_{t-u,u,x,f}.
\end{align}
	From \eqref{Fact:TO!}, we know that
\[
	\sup_{\epsilon > 0}\varlimsup_{t\to \infty}
	\sup_{u\in (0,\epsilon), x\in E, f \in \mathcal B(E,[0,\infty])}
	C^{15}_{t-u,u,x,f} < \infty
\]
	and that
\[
	\sup_{\epsilon > 0}\lim_{t\to \infty}
	\sup_{u\in (0,\epsilon), f \in \mathcal B(E,[0,\infty])}
       uC^{13}_{t-u,u,f} =0.
\]
	Thus,
\begin{equation}
	\label{Claim:PPV}
	P_u^\beta \Psi_0 V_{t-u} f(x)
	= \phi(x)\nu(V_tf) C^{16}_{t,u,f,x}
	\end{equation}
for some non-negative $C^{16}_{t, u,f,x}$ with $\sup_{\epsilon > 0} \varlimsup_{t\to \infty}
\sup_{u \in (0,\epsilon), x\in E, f \in \mathcal B(E,[0,\infty])} C^{16}_{t,u,f,x} < \infty$.

	Finally, we note that
\begin{align}
	&J_{s,\epsilon}V_tf(x) = \int_{s-\epsilon}^s P_{s-u}^\beta \Psi_0 V_{t+u} f(x)du
	= \int_0^\epsilon P_u^\beta \Psi_0 V_{t+s - u}f(x) du
	\\& \overset{\eqref{Claim:PPV}}= \int_0^\epsilon \phi(x) \nu(V_{t+s}f) C^{16}_{t+s,u,f,x}~du
	\leq \epsilon \phi(x)\nu(V_{t+s}f) \sup_{u\in (0,\epsilon)} C^{16}_{t+s,u,f,x}.
	\end{align}
	It is elementary to see that \[\lim_{\epsilon \to 0}\varlimsup_{t+s \to \infty}
	\sup_{x\in E, f \in \mathcal B(E,[0,\infty])}
		\Big(\epsilon \sup_{u\in (0,\epsilon)} C^{16}_{t+s,u,f,x}\Big) = 0.\]
		Combining the two displays above, we get the conclusion of Lemma \ref{prop:JVf}.
		\end{proof}

\subsection{Proof of Proposition \ref{prop:G}}\label{sec:G}
	Recall that for each $t\geq 0$, $\Gamma_t :=
		\mathscr L_{X_t;\mathbb P_\nu(\cdot | \|X_t\|>0)}$,
	the log-Laplace functional for $X_t$ under probability
	$\mathbb P_\nu(\cdot | \|X_t\|>0)$.
	For any unbounded increasing positive sequence $\mathbf t = (t_n)_{n\in \mathbb N}$, define $G^\mathbf t f = \varliminf_{n\to \infty} \Gamma_{(t_n)} f$.

	To prove Proposition \ref{prop:G}, we first prove two lemmas.

\begin{lem} \label{prop:Gtb:H1:H2:H3:H4}
	For any unbounded increasing positive sequence $\mathbf t = (t_n)_{n\in \mathbb N}$,
	$G^\mathbf t$ is a $[0,\infty]$-valued monotone concave functional on $\mathcal B(E,[0,\infty])$
	such that $G^{\mathbf t}(\infty \mathbf 1_E) = \infty$ and that
\begin{equation}
	1 - e^{-G^\mathbf t V_s f}
	= e^{s\lambda} (1-e^{- G^\mathbf t f}),
	\quad s\geq 0, f\in \mathcal B(E,[0,\infty]).
\end{equation}
\end{lem}
\begin{proof}
	Since $(\Gamma_t)_{t\geq 0}$ are $[0,\infty]$-valued functionals, so is $G^{\mathbf t}$.
	Also, from $\Gamma_t(\infty \mathbf 1_E) = \infty$ for all $t\geq 0$ we have that $G^{\mathbf t}(\infty \mathbf 1_E) = \infty$.
	We claim that $G^\mathbf t$ is monotone concave.
	In fact, for each $f \leq g$ in $\mathcal B(E,[0,\infty])$, we have
\begin{equation}
	G^{\mathbf t} f
	= \varliminf_{n\to \infty} \Gamma_{(t_n)} f
	\leq \varliminf_{n\to \infty} \Gamma_{(t_n)} g
	= G^{\mathbf t} g.
\end{equation}
	On the other hand, using Lemma \ref{Fact:CP!}, we have for all $t\geq 0$, $f\in \mathcal B(E,[0,\infty])$, $u,v \in [0,\infty)$, $r\in [0,1]$, it holds that
\begin{equation}
	\Gamma_t((ru+(1-r) v)f)
	\geq r \Gamma_t (uf) + (1-r) \Gamma_t (vf).
\end{equation}
	Therefore, for all $f\in \mathcal B(E,[0,\infty])$, $u,v \in [0,\infty)$, $r \in [0,1]$, we have
\begin{align}
	& G^{\mathbf t}((ru + (1-r)v)f)
	= \varliminf_{n \to \infty} \Gamma_{(t_n)}((ru + (1-r)v)f)
	\\&\geq \varliminf_{n\to \infty} (r\Gamma_{(t_n)} (uf) + (1-r)\Gamma_{(t_n)}(vf))
	\\&\geq r (\varliminf_{n\to \infty} \Gamma_{(t_n)} (uf)) + (1-r) (\varliminf_{n\to \infty} \Gamma_{(t_n)}(vf) )
	\\&= r G^{\mathbf t} (uf) + (1-r) G^{\mathbf t}(vf).
\end{align}
	
	Note that for any $t>0$ and $f\in \mathcal B(E,[0,\infty])$, it holds that
\begin{equation}\label{lem:Gfnv!}
	1 - e^{- \Gamma_t f}
	= \frac{ \mathbb P_\nu [ 1 - e^{- X_t(f)}]}{ \mathbb P_\nu (\|X_t\| > 0)}
	= \frac{ 1 - e^{- \nu(V_tf)} }{ 1 - e^{- \nu(v_t)}}.
\end{equation}
	Fix a function $f\in \mathcal B(E,[0,\infty])$.
	Thanks to \eqref{lem:nVn!} and \eqref{lem:Gfnv!}, we only need to consider the case $\nu(f) > 0$.
	In this case, by \eqref{lem:nVn!}, we have $\nu(V_tf)>0$ for each $t\geq 0$.
	Therefore, for any $s,t\geq 0$,
\begin{align}
	& 1 - e^{- \Gamma_t V_s f}
	\overset{\eqref{lem:Gfnv!}}= \frac{ 1 - e^{- \nu(V_{t+s} f)} }{ 1 - e^{- \nu(v_t)}}
	= \frac{ 1 - e^{- \nu(V_{t+s} f)} }{ 1 - e^{- \nu(V_tf)}} \frac{ 1 - e^{ - \nu(V_tf)}}{ 1 - e^{- \nu(v_t)}}
	\\ &  \label{eq:Gtb.5}\overset{\eqref{lem:Gfnv!}}= \frac{ 1 - e^{- \nu(V_{t+s} f)} }{ 1 - e^{- \nu(V_tf)}} ( 1 - e^{- \Gamma_t f}).
\end{align}
	Thus, for any $s\geq 0$,
\begin{align}
	& 1 - e^{- G^{\mathbf t} V_s f}
	= \varliminf_{n\to \infty} ( 1 - e^{- \Gamma_{(t_n)} V_s f})
	\overset{\text{\eqref{eq:Gtb.5}}}= \varliminf_{n\to \infty} \left( \frac{ 1 - e^{- \nu(V_{t_n+s}f)}}{ 1 - e^{- \nu(V_{(t_n)}f)}} (1 - e^{- \Gamma_{(t_n)} f}) \right)
	\\& = \left( \lim_{t \to \infty} \frac{ 1 - e^{- \nu(V_{t+s}f)}}{ 1 - e^{- \nu(V_{t}f)}} \right) \cdot \varliminf_{n\to \infty} (1 - e^{- \Gamma_{(t_n)} f} )
	= e^{s\lambda} (1 - e^{- G^{\mathbf t}f}),
\end{align}
	where the last equality follows from Proposition \ref{prop:Vf1::H1:H2::Y}, \eqref{eq:nVR}, and the fact that
	$
	(1-e^{-x}) /x \xrightarrow[x\to 0]{} 1.
	$
\end{proof}

\begin{lem} \label{prop:G*:H1:H2:H3:H4}
	Suppose that $r \in [\lambda,0)$.
	If $G_r$ is a $[0,\infty]$-valued monotone concave functional on $\mathcal B(E,[0,\infty])$
	such that $G_r(\infty \mathbf 1_E) = \infty$ and that
\begin{equation}
	1 - e^{-G_r V_s f}
	= e^{s r} (1 - e^{- G_r f}),
	\quad s\geq 0, f\in \mathcal B(E,[0,\infty]),
\end{equation}
	then for any unbounded increasing positive sequence $\mathbf t = (t_n)_{n\in \mathbb N}$,
\begin{equation}
	1 - e^{-G_r f} = (1 - e^{- G^\mathbf t f})^{r/\lambda}, \quad f \in \mathcal B(E,[0,\infty]).
\end{equation}
\end{lem}

\begin{proof}
	Let $(Q_t)_{t\geq 0}$ be the family of $[0,\infty)$-valued functionals on $\mathcal B(E,[0,\infty])$ given by
\[
	Q_tg
	:= e^{- r t}( 1 - e^{-G_r(gv_t)} ).
\]
	Note that, by \eqref{lem:sv2!}, $v_t(x)>0$ for all $x\in E$.
	It follows  from Proposition \ref{prop:Vf1::H1:H2::Y} that $v_t(x)<\infty$ for all $x\in E$ and all $t> T$.
	Thus $v_t(\cdot)$ is a $(0,\infty)$-valued function for all $t> T$.

	We claim that for any $u \in [0,1]$,
	$Q_t(u \mathbf 1_E)$ is non-increasing in $t\in (0,\infty)$.
	In particular, we can define the $[0,\infty]$-valued function
	$q(u):= \lim_{t\to \infty} Q_t(u \mathbf 1_E), u\in [0,1]$.
	In fact, note that $\mathbb P_{\delta_x}[e^{- X_s(uv_t)}] = e^{-V_s(uv_t)},x\in E, s,t>0, u \geq 0$.
	Lemma \ref{Fact:CP!} says that, for all $s,t > 0$ and $x\in E$, $u\mapsto V_s(uv_t)(x) $ is a $[0,\infty]$-valued concave function on $[0,\infty)$.
	Therefore, for $u\in [0,1]$, we have
\begin{equation}
	V_s(uv_t)
	\geq uV_s(v_t) + (1-u) V_s(0\cdot v_t)
	= uv_{s+t},
	\quad s,t > 0.
\end{equation}
	Using this, we get
\begin{align}
	& Q_{t+s}(u \mathbf 1_E)
	= e^{- r (t+s)} ( 1-e^{-G_r(uv_{t+s})} )
	\leq e^{- r(t+s)}( 1-e^{-G_r[V_s(uv_t)]} ) \\
	& = e^{-r t}( 1-e^{-G_r(uv_t)} )
	= Q_t(u \mathbf 1_E),
	\quad s,t > 0, u \in [0,1].
\end{align}

	We want to show that $q(u)= u^{r/\lambda}, u\in [0,1]$.
	In order to do this, we first show that
\begin{equation} \label{lem:qC!}
\begin{minipage}{0.9\textwidth}
	the function $q$ is non-decreasing and concave on $[0,1]$ with $q(1) = 1$.
	In particular, thanks to Lemma \ref{Fact:CR!}, $q$ is a continuous function on $(0,1]$.
\end{minipage}
\end{equation}
	In fact, from $G_r(\infty \mathbf 1_E) = \infty$ and $V_t(\infty \mathbf 1_E) = v_t$, we get
	\[
	Q_t( \mathbf 1_E)
	= e^{- r t} ( 1-e^{-G_r v_t} )
	= e^{- r t} e^{r t}( 1-e^{-G_r (\infty \mathbf 1_E)} )
	= 1,
	\quad t\geq 0.
	\]
	Therefore $q(1) = 1$.
	The above argument also says that $G_r v_t < \infty$ for each $t>0$.
	Now from the condition that $G_r$ is monotone concave, we have that for all $t>0$, the map $u \mapsto G_r(uv_t)$ is a non-decreasing and concave $[0,\infty)$-valued function on $[0,1]$.
	From Lemma \ref{lem:CE} we get that, for each $t> 0$, $u \mapsto Q_t(u \mathbf 1_E)$ is a $[0,\infty)$-valued, non-decreasing and concave function on $[0,1]$.
	Since the limit of concave functions is concave, we get \eqref{lem:qC!} by letting $t\to \infty$.

	We now show that
\begin{equation} \label{eq:GQ}
	q(u) = u^{r/\lambda},\quad u\in [0,1].
\end{equation}
To see this, note that for all $s\geq 0$, $t>T$ and $x\in E$, we have that
\begin{align}
	& e^{\lambda s}(\phi^{-1}v_t)(x)
	\overset{\text{Proposition \ref{prop:Vf2}}}= e^{\lambda s}\nu(v_{t})
		(1+ C^4_{t,x,\infty \mathbf 1_E})
	\\&\overset{\eqref{eq:nVR}}=\nu(v_{t+s})
	\exp\{-\lambda sC^{13}_{t,s,\infty \mathbf 1_E}\} (1+ C^4_{t,x,\infty \mathbf 1_E})
	\\&\overset{\text{Proposition \ref{prop:Vf2}}}= (\phi^{-1}v_{t+s})(x)
	(1+ C^4_{t+s,x,\infty \mathbf 1_E})^{-1} \exp\{-\lambda sC^{13}_{t,s,\infty \mathbf 1_E}\} (1+ C^4_{t,x,\infty \mathbf 1_E})
	\\& \label{eq:GQ.5}= (\phi^{-1}v_{t+s})(x) (1+C^{17}_{s,t,x}),
\end{align}
for some real $C_{s,t,x}^{17}$ with $\lim_{t\to \infty}\sup_{x\in E} |C_{s,t,x}^{17}| =0$.
	Thus, we know that for all $s\geq 0$ and $\epsilon >0$ there exists $T^1_{s,\epsilon}>0$ such that
	\begin{equation} \label{eq:GQ.6}
	1-\epsilon
	\leq \frac{e^{\lambda s}v_t(x)}{v_{t+s}(x)}
	\leq 1+\epsilon,
	 \quad x\in E, t> T^1_{s,\epsilon}.
	 \end{equation}
	From this we get that for all $s\geq 0, \epsilon > 0$, $t\ge T^1_{s,\epsilon}$, and $u\geq 0$,
\begin{align}
	& Q_{t+s}[ (1-\epsilon)u \mathbf 1_E ]
	= e^{-r (t+s)}( 1-e^{-G_r[(1-\epsilon)uv_{t+s}]} )
	\overset{\eqref{eq:GQ.6}}\leq e^{-r t} e^{-r s}( 1- e^{-G_r(ue^{\lambda s}v_t)} )
	\\&= e^{-r s}Q_t(ue^{\lambda s} \mathbf 1_E)
	\overset{\eqref{eq:GQ.6}}\leq e^{-r(t+s)}( 1-e^{-G_r[(1+\epsilon)uv_{t+s}]})
	\\&= Q_{t+s}[ (1+\epsilon)u \mathbf 1_E ].
\end{align}
	Letting $t\to \infty$ in the display above, we get that for all $s\geq 0$, $\epsilon > 0$ and $u$ satisfying $0 < (1 - \epsilon) u < (1+\epsilon)u < 1$, it holds that
\begin{align} \label{eq:GQ.9}
	& q((1-\epsilon)u)
	\leq e^{-r s}q(u e^{\lambda s})
	\leq q((1+\epsilon)u).
\end{align}
	Using \eqref{lem:qC!}, letting $\epsilon \to 0$ and then $u \uparrow 1$ in \eqref{eq:GQ.9}, we get that
\[
	q(1)
	=1
	= e^{- r s} q(e^{\lambda s}),
	\quad s \geq 0.
\]
	In other word, $q(u) = u^{r/\lambda}$ for $u\in (0,1]$.
	Finally noticing that $q$ is non-negative and non-decreasing on $[0,1]$, we also have $q(0) = 0$.

	We are now ready to finish the proof of Lemma \ref{prop:G*:H1:H2:H3:H4}.
	Fix an unbounded increasing positive sequence $\mathbf t=(t_n)_{n\in \mathbb N}$ and a function $f\in \mathcal B(E,[0,\infty])$, we only need to  prove that $1-G_r f =(1- G^{\mathbf t}f)^{r/\lambda}.$
	
	From the definition of $G^{\mathbf t} f$, we can choose a subsequence $\mathbf t'=(t'_n)_{n \in \mathbb N}$ of $\mathbf t$ such that for each $n\in \mathbb N$, we have $t'_n > T$ and
\begin{equation} \label{eq:vp.5}
		G^{\mathbf t}f = \Gamma_{t_n'}f + C^{18}_n
\end{equation}	
for some real $C^{18}_n$ (depending on both $f$ and $\mathbf t'$) such that $\lim_{n\to \infty} |C^{18}_n| =0$.	
	
Therefore, we have for any $n \in \mathbb N$,
\begin{align}
	& 1 - e^{- G^{\mathbf t}f}
		\overset{\eqref{eq:vp.5}}= 1 - e^{- \Gamma_{t'_n} f  - C^{18}_n}
		= (1 - e^{- \Gamma_{t'_n} f }) e^{- C^{18}_n} + (1 - e^{-C_n^{18}})
		\\&\overset{\eqref{lem:Gfnv!}}= \frac{1 - e^{- \nu( V_{(t_n')}f)}}{1- e^{- \nu(v_{(t_n')})}}  e^{ - C^{18}_n} + (1 - e^{-C_n^{18}})
		\\& \label{eq:vp.6}=
	\frac{\nu (V_{(t_n')} f)}{\nu(v_{(t_n')})}  (1+C^{19}_n) + (1 - e^{-C_n^{18}})
\end{align}
for some real $C^{19}_n$ with $\lim_{n\to \infty} |C^{19}_n| =0$, by Proposition
 \ref{prop:Vf1::H1:H2::Y} and the fact that $(1- e^{-x})/x \xrightarrow[x\to 0]{}1$.
	Thus
\begin{equation} \label{eq:vp.61}
	1 - e^{- G^{\mathbf t}f}
	 \overset{\text{ Proposition \ref{prop:Vf2}}}=  \frac{V_{(t_n')}f(x)}{v_{(t_n')}(x)}
\frac{1+C^4_{t_n',x,\infty \mathbf 1_E}}{1+C^4_{t_n',x,f}} (1+C^{19}_n) + (1 - e^{-C^{18}_n}).
\end{equation}
	It is elementary to see that
	\[
	\lim_{n\to \infty} \sup_{x\in E} \left|\frac{1+C^4_{t_n',x,\infty \mathbf 1_E}}{1+C^4_{t_n',x,f}} (1+C^{19}_n) -1 \right| = 0.
	\]
	Therefore, for any $\epsilon > 0$, there exists $N_\epsilon>0$ such that for any $n>N_\epsilon$,
\begin{equation} \label{eq:vp.62}
	\Big| \Big(\frac{1+C^4_{t_n',x,\infty \mathbf 1_E}}{1+C^4_{t_n',x,f}} (1+C^{19}_n)\Big)^{-1} -1 \Big| < \epsilon;
	\text{~and~} |1 - e^{-C^{18}_n}| < \epsilon.
\end{equation}
	Note from \eqref{Fact:BV!}, $0 \leq V_tf \leq v_t$ for each $t\geq 0$.
	It is elementary to verify from \eqref{eq:vp.61} and \eqref{eq:vp.62} that, for any $\epsilon>0$, $n>N_\epsilon$ and $x\in E$,
\[
	(1-\epsilon) \left((1 - e^{- G^{\mathbf t}f} - \epsilon)\vee 0\right)
	\leq \frac{V_{(t_n')}f(x)}{v_{(t'_n)}(x)}
	\leq (1+\epsilon) ( 1 - e^{- G^{\mathbf t}f} + \epsilon) \wedge 1.
\]
	Since $G_r$ is a monotone functional, we know that for each $t\geq 0$, $Q_t$ is also a monotone functional.
	This implies that  for any $\epsilon>0$ and $n>N_\epsilon$,
\begin{align} \label{eq:vp.7}
	&Q_{(t'_n)} \left[(1-\epsilon) \left((1 - e^{- G^{\mathbf t}f} - \epsilon)\vee 0\right) \mathbf 1_E \right]
	\leq Q_{(t'_n)}\left( \frac{V_{(t'_n)}f}{v_{(t'_n)}} \right)
	\\&\leq Q_{(t'_n)}\left[ \left((1+\epsilon) ( 1 - e^{- G^{\mathbf t}f} + \epsilon) \wedge 1\right) \mathbf 1_E \right].
\end{align}
	Note from the definition of $(Q_t)_{t\geq 0}$ and $G_r$, we always have for $t>T$ that
\[
	Q_t \left( \frac{V_tf}{v_t}  \right)
	= e^{- r t}( 1 - e^{- G_rV_tf}  )
	= 1- e^{- G_r f}.
\]
	Therefore, taking $n \to \infty$ in \eqref{eq:vp.7}, and using \eqref{eq:GQ}  we get that
\[
	\left((1-\epsilon) \left((1 - e^{- G^{\mathbf t}f} - \epsilon)\vee 0\right) \right)^{r/\lambda}
	\leq 1 - e^{- G_r f}
	 \leq \left((1 + \epsilon) (1 - e^{- G^{\mathbf t} f} + \epsilon)\wedge 1 \right)^{r/\lambda}.
\]
	Taking $\epsilon \to 0$, we get the desired result.
\end{proof}

\begin{proof}[Proof of Proposition \ref{prop:G}]
	Combining  Lemmas \ref{prop:Gtb:H1:H2:H3:H4} and \ref{prop:G*:H1:H2:H3:H4}
	(taking $r=\lambda$)
	with a sub-sub-sequence type argument, we can easily get the conclusion of Proposition \ref{prop:G}.
\end{proof}

\subsection{Proof of Proposition \ref{prop::GD:H1:H2:H3:H4::Y}}
\begin{proof}[Proof of Proposition \ref{prop::GD:H1:H2:H3:H4::Y}] \label{sec:GD}
	We first consider the case that $g = 0$ $\nu$-almost surely.
	From \eqref{eq:Vf1.1} and \eqref{asp:H2!}, we have
\begin{align}\label{eq:GD.1}
	&V_1 g_n(x)
	\leq P^\beta_1 g_n(x)
	\leq C^{20} \phi(x) \nu(g_n),
	\quad n \in \mathbb N, x\in E,
\end{align}
	where $C^{20} := \sup_{x\in E, f\in L_+^1(\nu)}e^{\lambda }(1+|H_{1,x,f}|)$.
	By the bounded convergence theorem, we have
\begin{equation} \label{eq:GD.11}
	\lim_{n\to \infty} \nu(g_n) = \nu(g) = 0.
\end{equation}
	On the other hand, from \eqref{eq:nuP.1}, we know that $ t\mapsto e^{-\lambda t}\nu(v_t)$ is a non-increasing $(0,\infty)$-valued continuous function on $(T,\infty)$.
	Since $\lambda <0$, we have
\begin{equation} \label{eq:GD.12}
\begin{minipage}{0.9\textwidth}
	$ t\mapsto \nu(v_t)$ is a strictly decreasing $(0,\infty)$-valued continuous function on $(T,\infty)$.
\end{minipage}
\end{equation}
	By Proposition \ref{prop:Vf1::H1:H2::Y}, we have
\begin{equation} \label{eq:GD.13}
	\lim_{t\to \infty}\nu(v_t) =0.
\end{equation}
	Using \eqref{eq:GD.11}, \eqref{eq:GD.12} and \eqref{eq:GD.13} we can see that there exist $n_0>0$ and a sequence $\{t_n: n>n_0\}$ of positive numbers such that
\begin{equation} \label{eq:GD.14}
	\lim_{n\to \infty} t_n = \infty
\end{equation}
	and that, for any $n>n_0$,
\begin{equation} \label{eq:GD.2}
2C^{20} \nu(g_n)
	\leq \nu(v_{t_n}).
\end{equation}
	It follows from Proposition \ref{prop:Vf2} that there exists $n_1 > n_0$ such that for all $n>n_1$ and $x\in E$,
\begin{equation} \label{eq:GD.25}
	\nu(v_{t_n})\leq 2\phi(x)^{-1} v_{t_n}(x).
\end{equation}
	Now, for any $n>n_1$ and $x\in E$, we have
\begin{align}
	& V_1g_n(x)
\overset{\eqref{eq:GD.1}}\leq C^{20} \phi(x)\nu(g_n)
	\overset{\eqref{eq:GD.2}}\leq \frac{1}{2}\phi(x)\nu(v_{t_n})
	\\\label{eq:GD.26} & \overset{\eqref{eq:GD.25}}\leq v_{t_n}(x).
\end{align}
	Therefore, for any $n>n_1$,
\begin{equation}
	1 - e^{- Gg_n}
	\overset{\text{\eqref{eq:G.0}}}= e^{- \lambda} (1- e^{- GV_1g_n})
	\leq e^{- \lambda} (1- e^{- G v_{t_n}})
	= e^{- \lambda} e^{\lambda t_n},
\end{equation}
	where in the inequality above we used \eqref{eq:GD.26} and the monotonicity of $G$ (Proposition \ref{prop:G}), and in the last equality, we used Proposition \ref{prop:G} with $f = \infty \mathbf 1_E$.
	Letting $n\to \infty$ in the display above, noticing \eqref{eq:GD.14} and the fact that $\lambda < 0$,
	we get the desired result in this case.

	We now consider the case that $g_n \downarrow g$ pointwisely
	where $\nu(g) > 0$.
	The monotonicity of $G$ (Proposition \ref{prop:G}) implies that $\lim_{n \to \infty} Gg_n$ exists and is greater
	than $G g$.
	So we only need to show that $\lim_{n\to \infty} Gg_n \leq Gg$.
	From Proposition \ref{prop:Vf2}, for any $\epsilon>0$ there exists $T^2_\epsilon >0$ such that for any $t\geq T^2_\epsilon, x\in E$ and $f\in \mathcal B(E,[0,\infty])$,
\begin{equation} \label{eq:GD.n1}
	(1-\epsilon)\phi(x) \nu(V_tf)\leq V_tf(x)
	\leq (1+\epsilon)\phi(x) \nu(V_tf).
\end{equation}
	Therefore, we have for any $\epsilon>0$, $t \geq T^2_\epsilon$, $x\in E$ and $f, h\in \mathcal B(E,[0,\infty])$ with $\nu(h) >0$ that
\begin{align}
	&V_tf(x)
	\overset{\eqref{eq:GD.n1}}\geq (1-\epsilon)\phi(x) \nu(V_tf)
	\\& \overset{\eqref{lem:nVn!}}= (1-\epsilon)\phi(x) \frac{\nu(V_tf)}{\nu(V_th)} \nu(V_th)
\overset{\eqref{eq:GD.n1}}\geq \frac{1-\epsilon}{1+\epsilon} \frac{\nu(V_tf)}{\nu(V_th)} V_th(x)
	\\ \label{eq:GD.n2} & \geq \left(\frac{1-\epsilon}{1+\epsilon} \frac{\nu(V_tf)}{\nu(V_th)} \wedge 1\right) V_th(x).
\end{align}
	Since $G$ is a monotone concave function (Proposition \ref{prop:G}), we know that for any $f\in \mathcal B(E,[0,\infty])$, $u \mapsto 1 - e^{-G(uf)}$ is a concave function on $[0,1]$ (Lemma \ref{lem:CE}); and therefore,
\begin{equation} \label{eq:GD.n3}
	1 - e^{- G(uf)}\geq u(1- e^{- Gf}) + (1-u) (1- e^{- G(0  \mathbf 1_E)}) = u(1- e^{- Gf}), \quad u \in [0,1].
\end{equation}
	Now we have for any $\epsilon>0$, $t \geq T^2_\epsilon$, $x\in E$ and $f, h\in \mathcal B(E,[0,\infty])$ with $\nu(h) >0$ that
\begin{align}
	&1 - e^{-Gf}
	\overset{\text{Proposition \ref{prop:G}}}= e^{-\lambda t}(1- e^{- GV_t f})
	\overset{\eqref{eq:GD.n2}}\geq e^{-\lambda t}\left(1- e^{- G\left( \left(\frac{1-\epsilon}{1+\epsilon} \frac{\nu(V_tf)}{\nu(V_th)} \wedge 1\right) V_th\right) }\right)
	\\ &\overset{\eqref{eq:GD.n3}}\geq e^{-\lambda t} \left(\frac{1-\epsilon}{1+\epsilon} \frac{\nu(V_tf)}{\nu(V_th)} \wedge 1 \right) \left(1- e^{- GV_th }\right)
	\overset{\text{Proposition \ref{prop:G}}}= \left(\frac{1-\epsilon}{1+\epsilon} \frac{\nu(V_tf)}{\nu(V_th)} \wedge 1\right) (1 - e^{- Gh}).
\end{align}
	Replacing $f$ by $g$, $h$ by $g_n$, and then taking $n\to \infty$, noticing that by monotone convergence theorem $\nu(V_tg_n) \xrightarrow[n\to \infty]{} \nu(V_tg)$, we get
\begin{equation}
	1 - e^{- Gg} \geq \frac{1 - \epsilon}{1+\epsilon}\lim_{n\to \infty}(1 - e^{- Gg_n}),
\end{equation}
	as desired (noticing $\epsilon > 0$ is arbitrary).
\end{proof}

\section{Proofs of Propositions \ref{prop:EQ}--\ref{prop:UC}} \label{sec:propsforthm2}

\subsection{Proof of Proposition \ref{prop:EQ}} \label{sec:EQ}
\begin{proof}[Proof of Proposition \ref{prop:EQ} (1)]	
	Denote by $G$ the functional given by Proposition \ref{prop:G}; and by $\mathbf Q_\lambda$ the Yaglom limit given by Theorem \ref{Theorem:Y:H1:H2:H3:H4}.
	By \eqref{eq:Y.000}, we know that $G$ is the log-Laplace functional of $\mathbf Q_\lambda$.
	Now note that for $t\geq 0$,
\begin{align}
	&(\mathbf Q_\lambda \mathbb P) (\|X_t\|>0)
	\overset{\text{\eqref{eq:OY.1}}}= \int_{\mathcal M_f(E)}(1-e^{-\mu(v_t)})\mathbf Q_\lambda (d\mu)
	\overset{\eqref{eq:Y.000}}= 1 - e^{-G v_t}
	\\&\overset{\text{Proposition \ref{prop:G}}}= e^{\lambda t}.  \label{eq:EQ.3}
\end{align}
	Therefore, we have that for all $f\in \mathcal B(E,[0,\infty])$ and $t \geq 0$,
\begin{align}
	&(\mathbf Q_\lambda \mathbb P)[1 - e^{-X_t(f)} | \|X_t\| > 0]
	\overset{\text{\eqref{eq:EQ.3}}}= e^{- \lambda t} (\mathbf Q_\lambda \mathbb P)[1- e^{-X_t(f)}]
	\\& \overset{\eqref{eq:BGD.2}}= e^{- \lambda t} \int_{\mathcal M_f(E)} (1- e^{-\mu(V_tf)} ) \mathbf Q_\lambda (d\mu)
	\overset{\eqref{eq:Y.000}}= e^{-\lambda t} (1 - e^{- G V_tf})
	\\& \overset{\text{Proposition \ref{prop:G}}}= 1 - e^{-Gf}
	\overset{\eqref{eq:Y.000}} =\int_{\mathcal M_f(E)} (1 - e^{-\mu(f)}) \mathbf Q_\lambda (d\mu).
\end{align}
	According to \cite[Theorem 1.17]{Li2011MeasureValued}, this says that
\[
	(\mathbf Q_\lambda \mathbb P)(\cdot |\|X_t\| > 0) = \mathbf Q_{\lambda}(\cdot),
	\quad t \geq 0.
\]
	Therefore $\mathbf Q_\lambda$ is a QSD of $X$.
	From \eqref{eq:EQ.3} and \eqref{eq:S.2}, its mass decay rate is $\lambda$.
\end{proof}

\begin{proof}[Proof of Proposition \ref{prop:EQ} (2)]
	Denote by $\gamma = r / \lambda \in (0,1)$.
	We first claim that there exists a $\mathbb Z_+$-valued random variable $\{Z;P\}$ with probability generating function $P[s^Z] = 1 - (1- s)^{\gamma}, s\in [0,1]$.
	To see this, we set
\[
	P(Z = n) = \frac{\gamma(1-\gamma ) \cdots (n-1-\gamma  )}{n!},
	\quad n \in \mathbb Z_+.
\]
	Using Newton's binomial theorem (see \cite[Exercise 8.22]{Rudin1976Principles}), we get
\[
	1 - (1 - s)^\gamma
	= \sum_{n = 1}^\infty \frac{\gamma (1-\gamma)\cdots (n-1-\gamma )}{n!} s^n,
	\quad s\in [0,1],
\]
	thus, such a random variable exists.
	
	Now let $\{(Y_n)_{n \in \mathbb N}; P\}$
	be an $\mathcal M^o_f(E)$-valued i.i.d. sequence
	with law of the Yaglom limit $\mathbf Q_\lambda$.
	Let $Z$ and $(Y_n)_{n\in \mathbb N}$ be independent of each other.
	Define the probability $\mathbf Q_r$ on $\mathcal M^o_f(E)$ as the law of the finite random measure $\sum_{n=1}^Z Y_n$.
	
	In the rest of this proof, we will argue that $\mathbf Q_r$ is a QSD of $X$ with mass decay rate $r$.
	To do this, we calculate that
\begin{align}
	&e^{- \mathscr L_{\mathbf Q_r} f}
	= P[ e^{-\sum_{n=1}^Z Y_n(f)} ] 	
	= P\left[P\left[ \prod_{n=1}^Z e^{-Y_n(f)} \middle | \sigma(Z)\right]\right]
	= P \left[ e^{-Z \cdot \mathscr L_{\mathbf Q_\lambda } f}\right]
	\\&= 1 - (1 - e^{- \mathscr L_{\mathbf Q_\lambda} f})^\gamma, \quad f\in \mathcal B(E,[0,\infty]).  \label{eq:EQ.4}
\end{align}
	Therefore, for each $t> 0$ and $f\in \mathcal B(E,[0,\infty])$, we have
\begin{align}
	&(\mathbf Q_r \mathbb P)\left[ 1 - e^{-X_t (f)} \middle|\|X_t\|>0 \right]
	= (\mathbf Q_r \mathbb P)(\|X_t\| >0)^{-1} \cdot (\mathbf Q_r \mathbb P) [1 - e^{- X_t(f)}]
	\\&\overset{\eqref{eq:BGD.2}, \eqref{eq:OY.1}}= (1 - e^{- \mathscr L_{\mathbf Q_r} v_t})^{-1}  (1 - e^{- \mathscr L_{\mathbf Q_r} V_tf})
	\overset{\text{\eqref{eq:EQ.4}}}= (1 - e^{- \mathscr L_{\mathbf Q_\lambda} v_t})^{-\gamma}(1 - e^{- \mathscr L_{\mathbf Q_\lambda} V_tf})^{\gamma}
	\\&\overset{\eqref{eq:BGD.2}, \eqref{eq:OY.1}}= (\mathbf Q_\lambda \mathbb P)\left[1 - e^{- X_t(f)}\middle| \|X_t\|>0\right]^{\gamma}
	\overset{\text{Proposition \ref{prop:EQ} (1)}}= (1 - e^{- \mathscr L_{\mathbf Q_\lambda} f})^{\gamma}
	\overset{\text{\eqref{eq:EQ.4}}}= 1 - e^{- \mathscr L_{\mathbf Q_r} f}.
\end{align}
	This proves that $\mathbf Q_r$ is a QSD.
	To see its mass decay rate is $r$, we calculate that for each $t\geq 0$,
\begin{align}
	&(\mathbf Q_r \mathbb P)(\|X_t\|>0)
	\overset{\eqref{eq:OY.1}}= 1 - e^{- \mathscr L_{\mathbf Q_r} v_t}
	\\& \overset{\text{\eqref{eq:EQ.4}}}= (1 - e^{- \mathscr L_{\mathbf Q_\lambda} v_t})^\gamma
	\overset{\eqref{eq:OY.1}}= (\mathbf Q_\lambda \mathbb P) (\|X_t > 0\|)^\gamma \overset{\text{Proposition \ref{prop:EQ} (1)}}= e^{ r t}. \qedhere
\end{align}
\end{proof}

\subsection{Proof of Proposition \ref{prop:CQ}} \label{sec:CQ}

\begin{proof}[Proof of Proposition \ref{prop:CQ} (1)]
	First observe that for any $t\geq 0$,
\begin{equation} \label{eq:CQ.1}
	e^{rt} = (\mathbf Q_r^*\mathbb P)(\|X_t\|>0)
	\overset{\eqref{eq:OY.1}}= 1 - e^{- \mathscr L_{\mathbf Q_r^*}(v_t)}.
\end{equation}
	According to Lemma \ref{Fact:CP!}, for any $t>0$, we know that $u\mapsto \mathscr L_{\mathbf Q_r^*} (uv_t)$ is a $[0,\infty]$-valued concave function on $[0,\infty)$.
	According to Lemma \ref{lem:CE}, for any $t>0$, we know that $u \mapsto 1 - e^{- \mathscr L_{\mathbf Q_r^*}(uv_t)}$ is a $[0,1]$-valued concave function on $[0, \infty)$.
	 In particular, we have for any $t>0$ and $u \in [0,1]$ that
\begin{equation} \label{eq:CQ.2}
		1 - e^{- \mathscr L_{\mathbf Q_r^*}(uv_t)} \geq u(1 - e^{- \mathscr L_{\mathbf Q_r^*}(1\cdot v_t)}) + (1-u) (1 - e^{- \mathscr L_{\mathbf Q_r^*}(0 \cdot v_t)})  = u(1 - e^{- \mathscr L_{\mathbf Q_r^*}(v_t)}).
\end{equation}
Recall that $T^1_{s,\epsilon}$ is the constant given in \eqref{eq:GQ.6}. Now for any $s>0, \epsilon > 0$ and $t > T^1_{s, \epsilon}$ we have
\begin{align}
	& e^{rs} \overset{\text{\eqref{eq:CQ.1}}}= \frac{1 - e^{-\mathscr L_{\mathbf Q_r^*}v_{t+s}}}{1 - e^{-\mathscr L_{\mathbf Q_r^*}v_{t}}}
	\overset{\eqref{eq:GQ.6}}\geq \frac{1 - e^{-\mathscr L_{\mathbf Q_r^*}( \frac{e^{\lambda s}}{1+\epsilon}v_t )}}{1 - e^{-\mathscr L_{\mathbf Q_r^*}(v_{t})}}
	\overset{\eqref{eq:CQ.2}}\geq \frac{e^{\lambda s}}{1+\epsilon}.
\end{align}
	Letting $\epsilon \to 0$, we get the desired result.
\end{proof}

\begin{proof}[Proof of Proposition \ref{prop:CQ} (2)]
	From the definition of QSD, we know that $\mathbf Q_r^*$ has no concentration
	on $\{\mathbf 0\}$.
	Therefore $\mathscr L_{\mathbf Q_r^*}(\infty  \mathbf 1_E) = \infty$.
	According to Lemma \ref{Fact:CP!}, we know that $\mathscr L_{\mathbf Q_r^*}$ is a monotone concave functional.
	Knowing that $\mathbf Q^*_r$ is a QSD for $X$ with mass decay rate $r$, it can be verified that for each $f\in \mathcal B(E,[0,\infty])$ and $t\geq 0$,
\begin{align}
	&1 - e^{-\mathscr L_{\mathbf Q_r^*}f}
	= (\mathbf Q_r^* \mathbb P) \left[1 - e^{-X_t(f)} \middle|\|X_t\|>0 \right]
	\\&= e^{-rt}(\mathbf Q_r^*\mathbb P)[1 - e^{- X_t(f)}]
	\overset{\eqref{eq:BGD.2}}= e^{-rt} \int_{\mathcal M_f(E)} (1 - e^{-\mu(V_tf)}) \mathbf Q_r^*(d\mu)
	\\&= e^{-rt} (1 - e^{- \mathscr L_{\mathbf Q_r^*}V_tf}).
	\qedhere
\end{align}
	
\end{proof}

\subsection{Proof of Proposition \ref{prop:UC}} \label{sec:UC}
\begin{proof}[Proof of Proposition \ref{prop:UC}]
	This is now obvious from Lemma \ref{prop:G*:H1:H2:H3:H4} and the fact that $Gf = \lim_{t\to \infty} \Gamma_t f$ for $f\in \mathcal B(E,[0,\infty])$ (Theorem \ref{prop:G}).
\end{proof}

\appendix\section{}
\subsection{Extended values} \label{sec:EV}
	In this paper, we often work with the extended non-negative real number system $[0,\infty]$ which consists of the non-negative real line $[0,\infty)$ and an extra point $\infty$.
	We consider $[0,\infty]$ as the one point compactification of $[0,\infty)$; and therefore, it is a compact Hausdorff space.
	We also make the following conventions that
\begin{itemize}
\item
	$x + \infty = \infty$ for each $x\in [0,\infty]$;
\item
	$x \cdot \infty = \infty$ for each $x\in (0,\infty]$;
\item
	$\frac{1}{\infty} = 0$; $\frac{1}{0} = \infty$; $e^{-\infty} =0$; $-\log 0 = \infty$.
\end{itemize}
	Note that $ \infty \cdot 0$ has no meaning, but we use the convention that $\infty \cdot 0 = 0$ when we are dealing with indicator functions.
	For example, we may write expression like
\begin{equation}
	h(x)
	= g(x) \cdot  \mathbf 1_{A} (x)+ \infty \cdot \mathbf 1_{E\setminus A}(x), \quad x\in E,
\end{equation}
	as a shorthand of
\begin{equation}
	x =
\begin{cases}
	g(x) & \text{if $x\in A$},
	\\ \infty & \text{if $x\in E\setminus A$}.
\end{cases}
\end{equation}

\subsection{Concave functionals}
	We say an $\mathbb R$-valued (or $[0,\infty]$-valued) function $f$ on a convex subset $D$ of $\mathbb R$ is concave iff
\[
   	f(rx+(1-r) y)
 	\geq r f(x) + (1-r) f(y),
 	\quad x,y \in D, r \in [0,1].
\]
	The following lemmas about concave functions are elementary, we refer our readers to \cite[Chapter 6]{Dudley2002Real} for more details.

\begin{lem} \label{Fact:CR!}
	If $f$ is a non-decreasing $\mathbb R$-valued concave function on $(a,b]$ where $a<b$ in $\mathbb R$, then $f$ is continuous on $(a,b]$.
\end{lem}

\begin{lem} \label{Fact:CP!}
	Suppose that $\{Z; P\}$ is a $[0,\infty]$-valued random variable.
	Define $L(u):= - \log P[e^{- u Z}]$ with $u \in [0,\infty)$, then $L$ is a $[0,\infty]$-valued concave function on $[0,\infty)$.
\end{lem}

\begin{lem} \label{lem:CE}
	Suppose that $g$ is a concave function on some convex subset $D$ of $\mathbb R$, then so is $q:= 1- e^{-g}.$
\end{lem}

\subsection{Continuity theorem for the Laplace functional of random measures}
	In this subsection, we discuss the continuity theorem for finite random measures on Polish space.	
	The following result is not new.
	We included it here for the sake of completeness.
	Let $E$ be a Polish space.
	Denote by $\mathcal M_f(E)$ the collection of all the finite Borel measures on $E$ equipped with the topology of weak convergence.
	According to \cite[Lemma 4.5]{Kallenberg2017Random}, $\mathcal M_f(E)$ is a Polish space.
\begin{lem} \label{fact:WC}
	Let $(\mathbf P_n)_{n\in \mathbb N}$ be a sequence of probabilities on $\mathcal M_f(E)$.
	Suppose that
	(1) for each $f \in \mathcal B_b(E,[0,\infty))$, limit $Lf := \lim_{n\to \infty}\mathscr L_{\mathbf P_n}f$ exists; and
	(2) for each $f_n \downarrow f$ pointwisely in $\mathcal B_b(E,[0,\infty))$, $Lf_n \downarrow Lf$.
	Then there exist an unique probability $\mathbf Q$ on $\mathcal M_f(E)$
	such that $(\mathbf P_n)_{n \in \mathbb N}$ converges weakly
to $\mathbf Q$ and  $\mathscr L_\mathbf Q = L$ on $\mathcal B_b(E, [0,\infty))$.
\end{lem}

\begin{proof}
	We say a $[0,\infty)$-valued functional $\Gamma$ on $\mathcal B_b(E,[0,\infty))$ is positive definite if
\[
	\sum_{i,j =1}^n a_i a_j \Gamma (f_i + f_j)\geq 0
\]
for any $\mathbb R$-valued list $(a_k)_{k = 1}^n$ and $\mathcal B_b(E,[0,\infty))$-valued list $(f_k)_{k = 1}^n$.
	It is proved in \cite[Theorem 3.3.3]{Dawson1992Infinitely} that for any $n \in \mathbb N$, $f\mapsto e^{- \mathscr L_{\mathbf P_n}f}$ is positive definite on $\mathcal B_b(E,[0,\infty))$.
	Therefore, $f \mapsto e^{- L f}$ is positive definite.
	Now from \cite[Corollary (A.6)]{Fitzsimmons1989Construction} and the condition (2), we know that there exists a sub-probability $\mathbf Q$ on $\mathcal M_f(E)$ such that
\begin{equation}	\label{eq:WC.1}
	\int_{\mathcal M_f(E)}e^{-\mu(f) } \mathbf Q(d\mu) = e^{-Lf}, \quad f \in \mathcal B_b(E,[0,\infty)).
\end{equation}
	Taking $f = 0 \cdot \mathbf 1_E$ in condition (1) we get that $L(0\cdot \mathbf 1_E) = 0$.
	This says that $\mathbf Q$ is a probability on $\mathcal M_f(E)$.
	Now condition (1) and \cite[Theorem 1.8]{Li2011MeasureValued} imply that $(\mathbf P_n)_{n \in \mathbb N}$ convergence to $\mathbf Q$ weakly.
	Finally, \eqref{eq:WC.1} implies that $\mathscr L_{\mathbf Q} = L$ on $\mathcal B_b(E,[0,\infty))$.
\end{proof}

\begin{acknowledgment*}
	We thank Zenghu Li and Leonid Mytnik for helpful conversations.
	We also thank the two referees for helpful comments on the first version of this paper.
\end{acknowledgment*}

\end{document}